\documentclass{amsart}
\usepackage{bbm}
\usepackage{graphicx}
\usepackage{verbatim}
\usepackage{mathtools}
\usepackage{color}

\usepackage[latin1]{inputenc}

\newtheorem{theorem}{Theorem}

\newtheorem{prop}{Proposition}
\newtheorem{lemma}{Lemma}[prop]

\newtheorem{assumption}{Assumption}

\def\N{{\mathbb N}}
\def\R{{\mathbb R}}
\def\Z{{\mathbb Z}}
\def\P{{\mathbb P}}
\def\E{{\mathbb E}}
\def\Z{{\mathbb Z}}

\newcommand{\diff}{\mathop{}\mathopen{}\mathrm{d}}
\newcommand\croc[1]{\left\langle #1\right\rangle}
\newcommand\ind[1]{\mathbbm{1}_{\left\{#1\right\}}}

\newcommand{\steq}[1]{\stackrel{\text{\rm #1.}}{=}}

\def\cal{\mathcal}

\def\card{\mathrm{Card}}

\title{Analysis of Large Urn Models with Local Mean-Field Interactions }

\address[W. Sun, Ph. Robert]{INRIA Paris, 2 rue Simone Iff, F-75012 Paris, France}
\author{Wen Sun}\thanks{The author's work has been supported by a public grant overseen by the French National Research Agency (ANR) as part of the ``Investissements d'Avenir'' program (reference: ANR-10-LABX-0098).}
\email{Wen.Sun@inria.fr}
\author{Philippe  Robert}
\email{Philippe.Robert@inria.fr}
\urladdr{http://team.inria.fr/rap/robert}

\date{\today}

\begin{document}

\begin{abstract}
  The stochastic models investigated in this paper describe the evolution of a set of  $F_N$ identical balls scattered into $N$ urns connected by an underlying symmetrical graph  with constant degree $h_N$. After some random  amount of time {\em all the balls} of any urn are redistributed  locally, among the $h_N$ urns of its neighborhood. The allocation of balls is done at random according to a set of weights which depend on the state of the system. The main original features of this context is that the cardinality $h_N$ of the range of interaction is not necessarily linear with respect to $N$ as in a classical mean-field context and, also, that the number of simultaneous jumps of the process is not bounded  due to the  redistribution of all balls of an urn at the same time. The approach relies on the analysis of the evolution of the local empirical distributions  associated to the state of urns located in the neighborhood of a given urn. Under convenient conditions, by taking an appropriate Wasserstein distance and by establishing several technical estimates for local empirical distributions, we are able to prove mean-field convergence results.

When the load per node goes to infinity, a convergence result for the invariant distribution of the associated McKean-Vlasov process is obtained for several allocation policies.  For the  class of power of $d$ choices policies,  we show that  the associated invariant measure has an asymptotic finite support property under this regime.  This result differs somewhat from the classical double exponential decay property usually encountered in the literature for power of $d$ choices policies.  
\end{abstract}

\maketitle

\vspace{-5mm}

\bigskip

\hrule

\vspace{-3mm}

\tableofcontents

\vspace{-1cm}

\hrule

\bigskip

\section{Introduction}\label{IntroSec}
The stochastic models investigated in this paper describe the evolution of a  set of  $N$ urns indexed by $i{\in}\{1,\ldots,N\}$ with $F_N$ identical balls. There is an underlying deterministic symmetrical  graph structure connecting the urns. The system evolves as follows. After some exponentially distributed  amount of time with mean $1$ {\em all the balls} of an urn with index $i{\in}\{1,\ldots,N\}$ say, are redistributed among a subset ${\cal H}^N(i)$ of urns in the neighborhood of urn $i$. An important feature of the model is that the allocation  of the balls into urns of ${\cal H}^N(i)$ is done at random according to a probability vector depending  on the number of balls  in the urns.

Quite general allocation schemes are investigated but two policies stand out because of their importance.  Their specific equilibrium properties are analyzed in detail in the paper. We describe them quickly. Assume that a ball of urn $i$ has to be allocated.
\begin{enumerate}
\item {\sc Random Policy.}\\
 The ball is allocated uniformly at random in one of the neighboring urns, i.e.\  in one  of the urns whose index is in ${\cal H}^N(i)$, this occurs  with probability $1/\card({\cal H}^N(i))$.
\item {\sc Power of $d$-Choices}\\
For this scheme, a subset of $d$ urns whose indices are in ${\cal H}^N(i)$ is taken at random, the ball is allocated to the urn having the least number of balls among these $d$ urns. Ties are broken with coin tossing. 
\end{enumerate}
Under some weak symmetry assumption and supposing that the state of the system can be represented by an irreducible finite state Markov process, at equilibrium the average number of balls per urn is $F_N/N$ whatever the allocation procedure is. The main problem considered in this paper concerns the distribution of the number of balls in a given urn when the number of urns is large. It may be expected that if the algorithm used to perform allocation is chosen conveniently, then there are few heavily loaded urns,  i.e.\  the probability of such event should be significantly small.  An additional desirable feature of such an algorithm  is that only a limited information should be required for the allocation of the balls. In our case this will be the knowledge of the occupancy of few urns among the neighboring urns. 

\bigskip

\begin{figure}[ht]
\center
\includegraphics[width=0.9\textwidth]{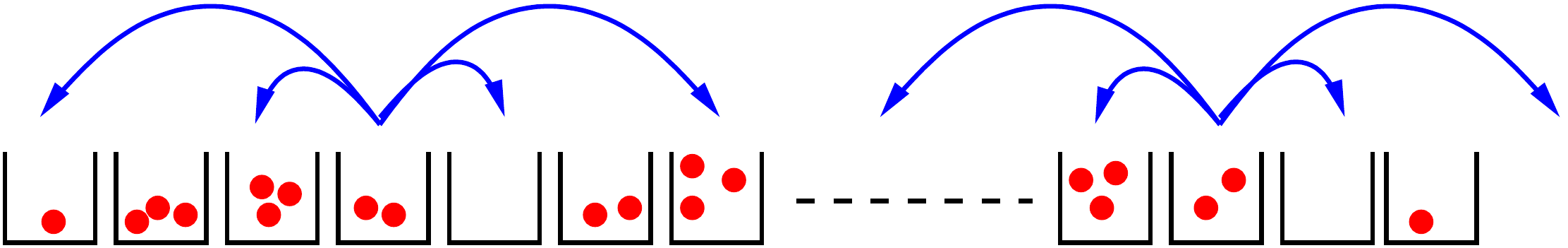}
\put(-85,45){\textcolor{blue}{${\cal H}^N(j)$}}
\put(-75,-10){\textcolor{blue}{$j$}}
\put(-260,45){\textcolor{blue}{${\cal H}^N(i)$}}
\put(-250,-10){\textcolor{blue}{$i$}}
\caption{Urn Model with Neighborhoods}
\end{figure}
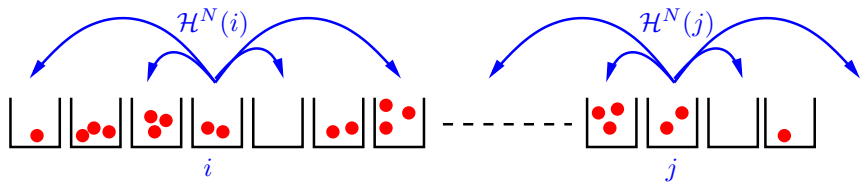

\bigskip

\subsection{Urn Models and Allocation Algorithms in the Literature}
These problems have important applications in several domains, in physics to describe interactions of particles and their non-equilibrium dynamics, see~Ehrenfest~\cite{Ehrenfest},  Godr\`eche and Luck~\cite{Godreche} and Evans and Hanney~\cite{Evans} for recent overviews. Urn models have been used in statistical mechanics as toy models to understand conceptual problems related to non-equilibrium properties of particle systems.  As a model they describe the evolution a set of particles moving from one urn to another one. Ehrenfest~\cite{Ehrenfest} is one of the most famous models: in continuous time models, each particle moves at random to another urn after an exponentially distributed amount of time. There may also be  underlying structure for the urns, a ball can be allocated to the connected urns at random according to some weights on the corresponding edges or, via a metropolis algorithm, associated to some energy functional. Due to their importance and simplicity, these stochastic models have been thoroughly investigated over the years: reversibility properties, precise estimates of fluctuations, hitting times, convergence rate to equilibrium, \ldots\  See, for example Diaconis~\cite{Diaconis}. For these models balls are moved one by one. Closely related to these models is the zero range process for which the corresponding rate is $f(x)$ for some general function $f$ on $\N$.  When $f(\cdot)$ is constant the jumps are in fact associated with the urns instead of the balls, as it is our case. See Evans and Hanney~\cite{Evans}.

These problems  are also considered in  theoretical computer science to evaluate the efficiency of algorithms to allocate tasks to processors in a large scale computing network for example.  See Karthik et al.~\cite{Karthik}, Maguluri et al.~\cite{Maguluri} and Sun et al.~\cite{SSMRS2}. A classical problem in this context is the assignments of $N$ balls into  $N$ urns. Balls are assumed to be allocated one by one. The constraints in this setting are of minimizing the maximum  of the number of balls in the urns with a reduced information on the state of the whole system. When balls are distributed randomly, it has been proved that the maximum number of balls in an urn is, with high probability, of the order of $\log N/\log\log N$. See Kolchin et al.~\cite{Kolchin}. By using an algorithm of the type power of $d$-choices as described above,  Azar et al.~\cite{Azar} has shown that the maximum is of the order of $\log\log N/\log d$. Hence, with a limited information on the system, only the state of $d$ urns is  required, the improvement over the random policy is striking. See also Mitzenmacher~\cite{Mitzenmacher}.

A related problem is of assigning the jobs to the queues of $N$ processing units working at rate $1$. The jobs are assumed to be arriving according to  a Poisson process with rate $\lambda N$. When the natural stability condition $\lambda{<}1$ holds, if the jobs are allocated at random, then, at equilibrium, it is easily shown that the tail distribution of the number of jobs at a given unit decreases exponentially. If the allocation follows a power of $d$ choices, Vvedenskaya et al.~\cite{Dobrushin} has shown that the corresponding tail distribution has a double exponential decay, i.e.\ of the type $x{\to}\exp(-\alpha_1\exp(\alpha_2 x))$ for some positive constants $\alpha_1$ and $\alpha_2$.

The initial motivation of this work is coming from a collaboration with computer scientists to study the efficiency of duplication algorithms in the context of a large distributed system. The urns are the disks of a set of servers and the balls are  copies of files on these disks.  The redistribution of balls  of an urn correspond to a disk crash, in this case,  the copies of its lost files are retrieved on other neighboring disks.    See Sun et al.~\cite{SSMRS2} for more details on the modeling aspects.

\subsection{Results}

\subsection*{Mean-Field Convergence}
For the stochastic models investigated in this paper there are $N$ urns and a total of $F_N$ balls with $F_N{\sim}\beta N$ for some $\beta{>}0$. The underlying graph is symmetrical with degree $h_N{=}\card({\cal H}^N(i))$, $1{\le}i{\le}N$. The sequence $(h_N)$ is assumed to converge to infinity.  The state of the system is represented by a vector $\ell{=}(\ell_i,1{\le}i{\le}N)$ describing the number of balls in the urns.  The main mathematical difficulties in the analysis  of the stochastic model  have two sources:
\begin{enumerate}
\item {\sc Multiple Simultaneous Jumps}.\\
When the exponential clock associated to urn $i$ rings,  then all its balls are redistributed   to other urns of the system. For this reason, there are   $\ell_i{\in}\{1,\ldots,F_N\}$ jumps occurring simultaneously.
\item {\sc Local Search}.\\
Let
\begin{equation}\label{LED1}
\Lambda^N_i=\frac{1}{h_N}\sum_{k\in{\cal H}^N(i)} \delta_{\ell_k}
\end{equation}
be {\em the local empirical distribution} around urn $i$, where $\delta_a$ is the Dirac mass at $a{\in}\N$. 
 When  urn $i$ has to  allocate a ball, it is sent to one of $h_N$ neighboring urns, $j{\in}{\cal H}^N(i)$, with a probability of the order of $\Psi(\Lambda_i^N,\ell_j)/h_N$, where $\Psi$ is a   functional on $M_1(\N){\times}\N$, $M_1(\N)$ is the set of probability distributions on $\N$.
\end{enumerate}
This interaction with only local neighborhoods  and the unbounded number of simultaneous jumps are at the origin of the main technical difficulties to establish a mean-field convergence theorem.  See the quite intricate evolution equation~\eqref{SDE2} for  these local empirical distributions below.  More specifically, this is, partially,  due to  a factor $\ell_i/h_N$ which has to be controlled in several integrands of the evolution equations. This is where unbounded jumps play a role. See Relation~\eqref{eqB} for example.  A convenient Wasserstein distance~\eqref{Wass2} is introduced for this reason.

It should be noted a classical mean-field analysis can be achieved when the sequences  $(h_N)$ has a linear growth, this is in fact close to a classical mean-field framework with full interaction. In this case the term $\ell_i/h_N$ is not anymore a problem since $\ell_i$ is bounded by $F_N{\sim}\beta N$, The same is true if the simultaneous jumps feature is removed by assuming for example that only a ball is transferred for each event. In this case a mean-field result can be established with standard methods for the model with  neighborhoods.

\subsubsection*{Literature} For an introduction to the classical mean-field approach, see Sznitman~\cite{Sznitman}. Specific results for power of $d$ choices policies when $h_N{=}N$ are presented in Graham~\cite{Graham}, see also Luczak and McDiarmid~\cite{Luczak}. Budhiraja et al.~\cite{Budhiraja2} considers also such a local interaction for a queueing network with an underlying graph which is, possibly, random, with  jumps of size $1$. In this setting when the graph is deterministic, the mean-field analysis can be carried out by using standard arguments. Andreis et al.~\cite{Pra}  investigates mean-field of jump processes associated to neural networks  with a large number of simultaneous jumps with an infinitesimal amplitude. Lu\c{c}on and Stannat~\cite{Lucon} establishes mean-field results  in a diffusion context when the mean-field interaction involves particles in a box whose size is linear with respect to $N$ and a spatial component with possible singularities. A related setting is also considered in M\"uller~\cite{Muller}.

\subsection*{Asymptotic Finite Support Property}

The mean field results show that, for a fixed asymptotic load $\beta$ per urn and under appropriate conditions,  the evolution of the state of the occupancy of a given urn is converging in distribution to some non-linear Markov process $(L_\beta(t))$. { A probability distribution $\pi$ on $\N$ will be said to be {\em invariant} for this process when the following property holds: if the distribution of the initial value $L_\beta(0)$ is $\pi$ then, for any $T{>}0$, the processes $(L_\beta(t{+}T))$ and $(L_\beta(t))$ have the same distribution.}  In particular the distribution of $L_\beta(t)$ is constant for all $t{\ge}0$. 

We show that, for all classes of allocations considered in this paper, when it exists, an invariant distribution of $(L_\beta(t)/\beta)$ has at a tail which is upper-bounded by an  exponentially decreasing function.  For the random algorithm  this is  an exact  exponentially decreasing tail distribution in fact.  This implies that a small, but significant, fraction of urns will have an arbitrarily large load. For this reason the performances  of random algorithms are weak in terms of occupancy of urns. Recall that these algorithms do not use any information to allocate balls into urns, the question is that if some minimal information is used,  can we improve the order of magnitude of the load of a given urn?

For  the power of $d$ choices algorithm,  we show that, when the average load $\beta$ is large, the invariant distribution of $(L_\beta(t)/\beta)$ is converging in distribution to a  distribution with a support in {\em the compact interval} $[0,d/(d{-}1)]$. When $d{=}2$, this is a uniform distribution on $[0,2]$. The striking feature is that, for an average load of $\beta$ per urn, at equilibrium,  the occupancy of a given urn is at most $\beta d/(d{-}1)$ with probability arbitrarily close to $1$.  In some way this can be seen as the equivalent of the double exponential decay property of Vvedenskaya et al.~\cite{Dobrushin} in this context.  Note that, contrary to the model analyzed in Vvedenskaya et al.~\cite{Dobrushin}, we do not have an explicit expression for the invariant distribution of $(L_\beta(t))$. It should be noted that this is an asymptotic picture for $N$ large and also $\beta$ large. Experiments seem to show nevertheless that, in practice, this is an accurate description.  See  for example Figure~2 of Sun et al.~\cite{SSMRS2} where $d{=}2$, $N{=}200$ and $\beta{=}150$ and the uniform distribution on $[0,300]$ is quite neat. Explicit  bounds  on error terms would be of interest but seem to be out of reach for the moment. 

\subsection{Outline of the Paper}
The stochastic model is presented in Section~\ref{SecMod}. Section~\ref{SecEvol} introduces the main evolution equations for the local empirical distributions and establishes the existence and uniqueness properties of the corresponding asymptotic McKean{-}Vlasov process. The main convergence results are proved in Section~\ref{MFsec} via several  technical estimates.   Section~\ref{SecInv} is devoted to the analysis of the invariant distribution of the McKean{-}Vlasov process. The finite support property is proved in this section. 

\subsection*{Acknowledgments}
The paper has benefited from several useful remarks from two anonymous reviewers. The authors are really grateful for the work they have done on the first version of the paper. 
\section{The Stochastic Model}\label{SecMod}
We give a precise mathematical description of our system with $N$ urns and $F_N$ balls  with the following scaling assumption
\begin{equation}\label{Scaling}
\beta=\lim_{N\to+\infty} \frac{F_N}{N},
\end{equation}
for some $\beta{>}0$. An index $N$ will be added on the important quantities  describing the system to stress the dependence on $N$ but not systematically to make the various equations involved more readable.

 In this section, we describe the topology of the class of graphs considered and the allocation policies which are investigated. Finally, the mathematical tools used to represent and analyze the evolution of the state of the system,  see Section~\ref{SecEvol}, are introduced.

\subsection{Graph Structure}\label{subsecgraph}

We first describe the topology of the system, i.e.\ the set of links between the $N$ urns. The urns connected to  a given urn~$i$  determine the set of possible destinations  when the balls of node~$i$ are redistributed. 

The $N$ nodes are labeled from $1$ to $N$ and a  set  ${\cal H}^N{\subset}\{2,\ldots,N\}$ which is assumed to satisfy the following property of symmetry
\begin{equation}\label{Hsym}
  k{\in}{\cal H}^N \Leftrightarrow ({N{+}2{-}k}\  {\rm mod}\ {N}) {\in}{\cal H}^N,
\end{equation}
with the convention used throughout the paper that ($0$ mod $N$) is $N$. 

The set ${\cal H}^N$  is the set of neighbors of node $1$.  The set ${\cal H}^N(i)$ of neighbors of a node  $i{\in}\{1,\ldots,N\}$  is defined by translation in the following way
\[
{\cal H}^N(i)=\left\{(i{+}j{-}1 {\rm\  mod\  } N):j{\in}{\cal H}^N\right\}.
\]
In particular  ${\cal H}^N(1){=}{\cal H}^N$, one denotes by $h_N$ the cardinality of ${\cal H}^N$. Relation~\eqref{Hsym}  gives the property that the associated graph is symmetrical that is, if $i{\in}{\cal H}^N(j)$ then $j{\in}{\cal H}^N(i)$.  We give some simple situations of graphs satisfying these assumptions. 

\bigskip

\noindent
{\bf Examples.}
\begin{enumerate}
\item  {\sc Full Graph:} ${\cal H}_{{\rm cc}}^N{\steq{def}}\{2,\ldots,N\}$. \\
  The interaction between the nodes induced by this topology corresponds to the classical mean-field setting with a full interaction.

  \medskip
  
The two following examples exhibit a partial interaction, the size of the neighborhood being linear in $N$ or, with a more limited interaction, of the order of $\log N$. 
\item[]
\item {\sc  Torus:} for $\alpha{\in}(0,1)$,
  \[
    {\cal H}_{\alpha}^N{\steq{def}}\{(1{+}j {\rm \ mod\ } N): j{\in}[{-} \alpha N,  \alpha N]{\cap}\Z\}{\setminus}\{1\}.
    \]
\item[]
\item {\sc Log-Torus:} for $\delta{>}0$,
\[
    {\cal H}_{\log}^N{\steq{def}}\{(1{+}j {\rm \ mod\ } N): j{\in}\left({-}\lfloor \delta\log N\rfloor, \lfloor\delta\log N\rfloor\right){\cap}\Z\}{\setminus}\{1\}.
    \]
\end{enumerate}
\medskip

We have chosen that $1{\not\in}{\cal H}^N$ and, consequently, $i{\not\in}{\cal H}^N(i)$. For the allocation process it expresses the fact that a ball of a given urn cannot be re-allocated to this urn when the ball is redistributed. Our results in the following do not need this assumption in fact.  It just simplifies some steps of the proofs, to compute the previsible increasing processes of some martingales in particular.

\newpage

\setcounter{assumption}{19}
\bigskip

\hrule

\begin{assumption}[Topology]\label{CondT}
  \begin{enumerate}
    \item[]
\item The sequence of the degrees of nodes  $(h_N){\steq{def}}(\card({\cal H}^N))$ is converging to infinity.
\item[]
\item  The interaction set of node $i$, $1{\le}i{\le}N$, is defined as
\begin{equation}\label{range}
{\cal A}^N(i){\steq{def}} \left\{\rule{0mm}{4mm}j{\in}\{1,\ldots,N\}: {\cal H}^N(j){\cap}{\cal H}^N(i){\not=}\emptyset\right\}
\end{equation}
and its cardinality satisfies
    \[
\lim_{N\to+\infty}    \frac{{\card}({\cal A}^N(1))}{h_N^2}=0.
    \]
  \end{enumerate}
\end{assumption}

\bigskip

\hrule

\bigskip

{The set ${\cal A}^N(i)$ is the set of nodes which interact with a node in the neighborhood of $i$. Technically, this subset appears  naturally in the evolution equation~\eqref{SDE2} of the process $(\Lambda_i^N(t))$, the local empirical distribution at a given node $i{\in}\{1,\ldots,N\}$ defined by Relation~\eqref{LED1}.  The condition on its cardinality is used in Lemma~\ref{lemM} to prove that the martingale part of these equations is asymptotically negligible. See also Relation~\eqref{crocM}. 

  Assumption~\ref{CondT} is clearly satisfied for the examples described above, note that in these cases ${\card}({\cal A}^N(1)){\le}  C h_N$, for some constant $C$.
  }
\subsection{Allocation Algorithms}\label{subsecallocation}
A set of $F_N$ balls are scattered in the nodes, the state of the system is thus described by a vector $\ell{=}(\ell_i)\in{\cal S}_N$, with
\[
{\cal S}_N=\{(\ell_i)\in\N^N: \ell_1{+}\ell_2{+}\cdots{+}\ell_N{=}F_N\},
\]
for $i{\in}\{1,\ldots,N\}$, $\ell_i$ is the number of balls in urn $i$. For each $\ell{\in}{\cal S}_N$, one associates a probability vector $P^N(\ell){=}(p_{i}^N(\ell),1{\le}i{\le}N)$  with support on ${\cal H}^N$, i.e.\  $p^N_{i}{=}0$ if $i{\not\in}{\cal H}^N$. The vector $P^N(\ell)$ is in fact  the set of weights associated to urn $1$ in state $\ell$ and  $p_{i}^N(\ell)$ is the probability that a given ball of urn  $1$ is allocated to  urn $i$. As for the topology, we define the vector of weights for the other urns by translation. 

For $i{\in}\{1,\ldots, N\}$ a probability vector $P_i^N(\ell){=}(p_{ij}^N(\ell))$ with support on ${\cal H}^N(i)$ is defined  by
\begin{equation}\label{pij}
  p_{ij}^N(\ell){=}p_a^N(y),
\end{equation}
with $a{=}(j{-}i{+}1 {\rm \ mod\ }N)$  and $y{=}(\ell_{(i+k-1\ {\rm mod} N)}, 1{\le}k{\le}N)$, for any $j{\in}\{1,\ldots,N\}$. In particular $P_1^N(\ell){=}P^N(\ell)$ and $P_i^N(\cdot)$ is  the vector $P^N(\cdot)$ ``centered'' around node $i$. The quantity $p_{ij}^N(\ell)$ is the probability that, in state $\ell$, a ball of urn $i$ is allocated to urn $j$. 

 The dynamics are as follows: After an exponentially distributed amount of time with mean $1$,  the balls of an  urn  $i$, $1{\le}i{\le}N$, are distributed  into the neighboring urns  in the subset  ${\cal H}^N(i)$ one by one, according to some policy depending  on the state $\ell{=}(\ell_j)$ of the system just before the event.  Each of the $\ell_i$ balls of the $i$th urn is distributed on ${\cal H}^N(i)$ according to the probability vector $P_i(\ell)$ and the corresponding $\ell_i$ choices are assumed to be independent.  As it will be seen our asymptotic results hold under general assumptions, the following cases will be discussed due to their  practical importance.

 \noindent
{\bf Examples.}

\begin{enumerate}
\item {\sc Random Algorithm.}\\
  Balls of a given urn $i{\in}\{1,\ldots,N\}$ are sent uniformly at random into an urn with index in ${\cal H}^N(i)$. This is the simplest policy which is used in large distributed systems. 
This corresponds to the case where
\begin{equation}\label{RandAlg}
  p_{j}^N(\ell){=}{1}/{h_N},\quad j{\in}{\cal H}^N.
\end{equation}
\item[]\item {\sc Random Weighted Algorithm.}\\
  Each  ball is sent into  urn $j{\in}{\cal H}^N(i)$  with a probability proportional to $W(\ell_j)$, that is 
\begin{equation}\label{RandWeiAlg}
  p_{j}^N(\ell)=W(\ell_j)\left/\displaystyle\sum_{k{\in}{\cal H}^N}W(\ell_k),\quad j{\in}{\cal H}^N\right.,
\end{equation}
where  $W$ is some function on $\N$ with values in $(c,C)$  with $0{<}c{<}C$.
\item[]\item {\sc Power of $d${-}choices,  $d{\ge}2$.}\\
  For each ball, $d$ urns  are chosen at random in ${\cal H}^N(i)$, the ball is allocated  to the urn of this subset having the minimum number of balls. Ties are broken with coin tossing. {The probability of assigning a ball to an urn with at least $m$ balls, $m{\in}\N$, is therefore
  \[
\binom{\sum_{k\in{\cal H}^N(i)}\ind{\ell_k\ge m}}{d}\left/\displaystyle\binom{h_N}{d}\right.,
  \]
with the convention that $\binom{n}{d}{=}0$ if $n{<}d$. Hence, by taking into account the ties, for $j{\in}{\cal H}^N$, }
\begin{multline}\label{PowAlg}
p_{j}^N(\ell)=\\\left[\binom{\sum_{k\in{\cal H}^N}\ind{\ell_k\ge \ell_j}}{d}-\binom{\sum_{k\in{\cal H}^N}\ind{\ell_k>\ell_j}}{d} \right]\left/\displaystyle\binom{h_N}{d}\sum_{k\in{\cal H}^N}\ind{\ell_k=\ell_j}\right.,
\end{multline}
since $p_{j}^N(\ell)$ is the probability of assigning a ball to a {\em given} urn of ${\cal H}$ with $\ell_j$ balls.
\end{enumerate}

\bigskip

{For simplicity we have chosen to consider only the state of the system just before the jump for all the balls which have to be moved. The state of the system could be also updated after each ball allocation and, consequently the vector $p_i^N(\cdot)$, the dynamical description would then be more intricate to express.}

\medskip

It should be noted that the exponential clock associated to the $i$th urn gives  simultaneous jumps of $\ell_j$ coordinates with high probability if $N$ is large. In particular the magnitude of a possible jump is not bounded which leads to significant technical complications to prove a mean-field result as it will be seen. 

\bigskip

\noindent
We now introduce the main assumption on the allocation policies investigated. It describes the asymptotic behavior of the probability vector $(P^N(\cdot))$ when $N$ is large with a key functional $\Psi$ which is playing an important role in the subsequent analysis. The non-linear part of the dynamic of the associated McKean{-}Vlasov process is essentially expressed with this functional. See Relation~\eqref{McKV} of Section~\ref{SecEvol} for example.

\bigskip

\hrule

\medskip

\setcounter{assumption}{0}
\begin{assumption}[Allocation Algorithm]\label{CondA}
  \begin{enumerate}
    \item[]
\item {\em  The sequence $(P^N(\cdot)){=}(p_i^N(\cdot))$  satisfies the relation
\begin{equation}\label{PPsi}
\lim_{N\to+\infty}\sup_{\substack{i\in{\cal H}^N\\\ell\in{\cal S}_N}}   \left| h_Np_i^N(\ell){-}\Psi\left(\frac{1}{h_N}\sum_{j{\in}{\cal H}^N} \delta_{\ell_j},\ell_i\right)\right|=0,
\end{equation}
where $\Psi$ is a non-negative bounded function on $M_1(\N){\times}\N$ such that, for any $\sigma{\in}M_1(\N)$,
\begin{equation}\label{Psimass}
\int_\N \Psi(\sigma,x)\,\sigma(\diff x)=1.
\end{equation}
\item  There exist constants $C_{\Psi}$, $D_{\Psi}{>}0$ such that  
\begin{equation}\label{LipPsi}
\displaystyle  \begin{cases}
\displaystyle    \left|\Psi(\sigma,l){-}\Psi(\sigma,l')\right|\leq C_\Psi \left|l{-}l'\right|,\\
\displaystyle        \left|\Psi(\sigma,l){-}\Psi(\sigma',l)\right|\leq    D_\Psi \left\|\sigma{-}\sigma'\right\|_{{\rm tv}}
 \end{cases}
\end{equation}
holds for any $(\sigma,l)$ and $(\sigma',l'){\in}M_1(\N){\times}\N$.}
\end{enumerate}
\end{assumption}

\medskip

\hrule

\bigskip

\noindent
The set $M_1(\N)$ is the space of probability distributions on $\N$  and $\|{\cdot}\|_{{\rm tv}}$  is the total variation norm defined by, if $\mu_1$, $\mu_2{\in}M_1(\N)$
\begin{multline}\label{TV}
  \|\mu_1{-}\mu_2\|_{{\rm tv}}=\sup\left(\rule{0mm}{4mm} |\croc{\mu_1,f}{-}\croc{\mu_2,f}|: f{:}\N\to\{0,1\}\right)\\
  =\inf_{\pi{\in}[\mu_1,\mu_2]}\left(\int_{\N^2} \ind{x{\not=}y}\pi(\diff x,\diff y) \right),
\end{multline}
where $[\mu_1,\mu_2]$ is the set of couplings $Q$ associated to $\mu_1$ and $\mu_2$,  elements $Q{\in}M_1(\N^2)$ such that $Q(\diff x,\N){=}\mu_1$ and $Q(\N,\diff y){=}\mu_2$.
See Proposition~4.7 of Levin et al.~\cite{Peres}.

\noindent

Relation~\eqref{Psimass} is a conservation of mass condition, all balls  are reallocated. As it will be seen in the investigation of mean-field convergence, the specific non-linear term in the limiting dynamical system is expressed with the functional $\Psi$, Condition~\eqref{LipPsi} is just a classical Lipschitz condition as it is quite common in such a setting.

\bigskip

\noindent
    {\bf Examples.}

    \medskip

    \noindent
{\sc Random Weighted Algorithm.}\\
  Assumption~\ref{CondA} is satisfied with the function $\Psi$ given by 
  \[
  \Psi_{\rm cc}(\sigma,l)=W(l)\left/\int W(x)\,\sigma(\diff  x)\right.
  \]
  if the range of $W$ is in $[c,C]$, for some positive constants $c$ and $C$. 

  \noindent
{\sc Power of $d${-}choices.}\\
It is not difficult to check that  Assumption~\ref{CondA} is satisfied with
  \[
  \Psi_{\rm pc}(\sigma,l){=}\frac{(\sigma([l,{+}\infty))^d{-}\sigma((l,{+}\infty))^d}{\sigma(\{l\})},
    \]
    with the convention that $0/0{=}0$. 

\subsection{Stochastic Representation of the Dynamics of Allocation}
In order to use a convenient stochastic calculus to study these allocation algorithms, one has to introduce marked Poisson processes.   See Chapter~5 of Kingman~\cite{Kingman}  for an introduction on marked  Poisson point processes. 

We define the space of marks ${\cal M}{=}[0,1]^{\N}$, a mark $u{=}(u_k){\in}{\cal M}$ associated to an urn  $i{\in}\{1,\ldots,N\}$  describes how the balls of this  urn are allocated in the system: If the state is $\ell{=}(\ell_j)$, assuming that the balls are indexed by $1{\le}k{\le}\ell_i$, the $k$th ball  is allocated to urn $j{\in}\{1,\ldots,N\}$ if
\begin{equation}\label{Im}
u_k\in I_{ij}(\ell){\steq{def}}\left[\sum_{n=1}^{j-1} p_{in}^N(\ell),  \sum_{n=1}^{j} p_{in}^N(\ell)\right),
\end{equation}
if $u_k$ is a uniform random variable on $[0,1]$, this occurs with probability $p_{ij}^N(\ell)$. 
For  $i$, $j{\in}\{1,\ldots,N\}$, we introduce the family of mappings $Z_{ij}$ on ${\cal M}{\times}{\cal S}_N$ defined by, for  $u{\in}{\cal M}$ and $\ell{\in}{\cal S}_N$, 
\begin{equation}\label{Zu}
Z_{ij}(u,\ell)\steq{def}\sum_{k=1}^{\ell_i} \ind{u_k\in I_{ij}(\ell)},
\end{equation}
the quantity $Z_{ij}(u,\ell)$ is the number of balls of urn $i$ which are allocated  to urn $j$ if the $i$th urn is emptied when the system is in state $\ell$ and with mark $u$. If $U{=}(U_k)$ is an i.i.d.\  sequence of uniform random variables on $[0,1]$ then, clearly, for $i{\in}\{1,\ldots, N\}$ and $j\in{\cal H}^N(i)$, and $\ell{\in}{\cal S}_N$, 
\begin{equation}\label{Mij}
(Z_{ij}(U,\ell), j{\in}\{1,\ldots,N\})\steq{dist} B_{i}(\ell)(\diff z_1,\ldots,\diff z_N),
\end{equation}
where $B_i(\ell)$ is a multinomial distribution with parameters $\ell_i$ and $p^N_{i1}(\ell)$, \ldots,  $p^N_{iN}(\ell)$, in particular
\begin{equation}\label{Bij}
 Z_{ij}(U,\ell)\steq{dist} B_{ij}(\ell)(\diff z),
\end{equation}
$B_{ij}(\ell)$ is a  binomial distribution with parameter $\ell_i$ and $p_{ij}^N(\ell)$. 

Let $\overline{{\cal N}}$ be a marked Poisson point process on $\R_+{\times}{\cal M}$ with intensity measure
\[
\diff t{\otimes}\prod_{i=1}^{+\infty}\diff u_i
\]
on $ \R_+{\times}{\cal M}$. Such a process can be represented as follows.  If ${\cal N}{=}(t_n)$ is a standard Poisson process on $\R_+$ with rate $1$ and $((U_k^{n})$, $n{\in}\N)$  is a sequence of  i.i.d.\ sequences of uniform random variables on $[0,1]$, then the point process $\overline{{\cal N}}$  on $\R_+{\times}[0,1]^{\N}$ can be defined by 
\[
\overline{{\cal N}}{=}\sum_{n\geq 1}\delta_{(t_n,(U_k^{n}))},
\]
where $\delta_{(a,b)}$ is the Dirac mass at $(a,b)$.  If $A{\in}{\cal B}(\R_+{\times}{\cal M})$ is a Borelian subset of $\R_+{\times}{\cal M}$,
\[
\overline{{\cal N}}(A)=\int_{A}\,\overline{{\cal N}}(\diff t,\diff u)
\]
is the number of points of $\overline{{\cal N}}$ in $A$. We denote  by ${\cal N}$ the point process on $\R_+$  defined by the first coordinates  of the points of  $\overline{{\cal N}}$, i.e.\   ${\cal N}(\diff t){=}\overline{{\cal N}}(\diff t, {\cal M})$, ${\cal N}$ is a Poisson process on $\R_+$ with rate $1$. 
We denote by $\overline{\cal N}_i$, $i{\in}\N$, i.i.d.\  marked Poisson point processes with the same distribution as $\overline{\cal N}$.

The martingale property mentioned in the following  is associated  to the filtration $({\cal F}_t)$ defined as follows, ${\cal F}_0$ is the $\sigma$-field associated to the initial state of the urn process and, for $t{>}0$,
\[
{\cal F}_t=\sigma\croc{{\cal F}_0, \overline{\cal N}_i([0,s]{\times}B): i\in\N, s{\le}t, B{\in} {\cal B}\left({\cal M}\right)}.
\]
We  recall an elementary result concerning the martingales associated to marked Poisson point processes. It is used throughout the paper.  See Section~4.5 of Jacobsen~\cite{Jacobsen}, see also Last and Brandt~\cite{Last} for more details. 
\begin{prop}\label{MPPMart}
For $1{\le}i{\le}N$, if $h$ is a Borelian function on $\R_+{\times}{\cal M}$  c\`adl\`ag on the first coordinate and such that 
\[
\int_{[0,t]\times {\cal M}}h(s,u)^2\,\diff s\diff u{<}{+}\infty, \quad \forall t\geq 0,
\]
where $\diff u$ denotes the product of Lebesgue measures on ${\cal M}{=}[0,1]^{\N}$, then the process
\[
(M(t))\steq{def}\left(\int_{[0,t]\times {\cal M}} h(s{-},u))\overline{\cal N}_i(\diff s, \diff u)-\int_{[0,t]\times {\cal M}}h(s,u)\,\diff s\diff u\right)
  \]
  is a square integrable martingale with respect to the filtration $({\cal F}_t)$, its previsible increasing process is given by
  \[
\left(\croc{M}(t)\right)=\left(\int_{[0,t]\times{\cal M}}h(s,u)^2\,\diff s\diff u\right).
  \]
\end{prop}
\noindent
We conclude this section with some notations which will be used throughout the paper. 

\subsection{Wasserstein Distances}
Throughout the paper  $M_1(X)$ denotes the set of probability distributions on the set $X$.  If $\mu{\in}M_1(\N)$ and $f{:}\N\to \R$,
\[
\croc{\mu,f}\steq{def}\int_\N f(x)\,\mu(\diff x)=\sum_{k\in\N} f(k)\mu(\{k\}),
\]
provided that the latter term is well defined. The function $I$ on $\N$ is the identity function, $I(x){\steq{def}}x$, $x{\in}\N$. 

The space $M_1(\N)$  endowed with the total variation norm is a separable Banach space.
For $p{>}0$, we define the Wasserstein distance, 
\begin{equation}\label{Wp}
  W_p(\sigma_1,\sigma_2)=\inf_{ Q{\in}[\sigma_1,\sigma_2]}\left(\int_{\N^2}|x{-}y|^p\,Q(\diff x{,}\diff y)\right)^{1/p},
\end{equation}
 where the set $[\sigma_1,\sigma_2]$ of couplings of $\sigma_1$ and $\sigma_2$ is defined below Relation~\eqref{TV}.

For $T{\ge}0$, we will denote  by ${\cal C}([0,T],M_1(\N))$ (resp. ${\cal D}([0,T],M_1(\N))$)  the space of continuous (resp.  c\`adl\`ag)  functions with values in $M_1(\N)$. We denote by $d_T(\cdot,\cdot)$ the distance associated to the topology of ${\cal D}_T{\steq{def}}{\cal D}([0,T],M_1(\N))$ as defined in Section~3.5 of Ethier and Kurtz~\cite{Ethier},  in this way  $({\cal D}_T,d_T)$ is a complete separable metric space. 

We introduce the  Wasserstein metric ${\cal W}_T$ on the corresponding stochastic process, on $M_1({\cal D}_T)$, the space of c\`adl\`ag processes with values in $M_1(\N)$. Let $P_1$ and $P_2{\in}M_1({\cal D}_T)$,
\begin{equation}\label{Wass}
{\cal W}_T(P_1,P_2){\steq{def}}\inf_{Q\in[P_1,P_2]}\int_{{\cal D}_T^2}\left(d_T(\Lambda_1,\Lambda_2){\land}1\right) \,Q(\diff\Lambda_1,\diff\Lambda_2),
\end{equation}
where, for $a$, $b{\in}\R$, $a{\land}b{=}\min(a,b$ and $[P_1,P_2]$ is, similarly to Definition~\eqref{TV}, the set of couplings associated to  distributions $P_1$ and $P_2$,  i.e.\   $Q{\in}[P_1,P_2]$ is an element of $M_1({\cal D}_T^2)$ with  marginals are $P_1$ and $P_2$ respectively. The metric space $(M_1({\cal D}_T),{\cal W}_T)$ is complete and separable.  See Section~3 of Dawson~\cite{Dawson} for a more specific presentation of measure-valued stochastic processes.

We will use the stronger Wasserstein distance $\overline{{\cal W}}_T$ to establish our convergence results,
\begin{equation}\label{Wass2}
\overline{{\cal W}}_T(P_1,P_2){\steq{def}}\inf_{Q\in[P_1,P_2]}\int_{{\cal D}_T^2} \left(\sup_{s\leq T}\|\Lambda_1(s){-}\Lambda_2(s)\|_{{\rm tv}}\right) \,Q(\diff\Lambda_1,\diff\Lambda_2).
\end{equation}

\section{Evolution Equations}\label{SecEvol}
In this section, we introduce the stochastic processes used to represent the evolution of the state of the system as well as the stochastic differential equations (SDE) they satisfy.

The state of the system at time $t$ is denoted by a c\`adl\`ag process $(L^N(t))$, with 
\[
(L^N(t))=\left(\left(L_i^N(t),1{\leq}i{\leq}N\right)\right)\in{\cal S}_N,
\]
$L_i^N(t)$ is the number of balls in urn $i$ at time $t$.
One defines the {\em local empirical distribution} at $i$, $1{\le}i{\le}N$, at time $t{\ge}0$ by 
\begin{equation}\label{LED}
\Lambda^N_i(t)\steq{def}\frac{1}{h_N}\sum_{j\in{\cal H}^N(i)}\delta_{L_j^N(t)},
\end{equation}
the global empirical distribution is, classically, 
\begin{equation}\label{ED}
\Lambda^N(t)\steq{def}\frac{1}{N}\sum_{j=1}^N\delta_{L_j^N(t)}=\frac{1}{N}\sum_{j=1}^N \Lambda^N_j(t).
\end{equation}

\medskip
The process $(L_i^N(t))$  can be represented as the solution of the following SDE, for $1{\le}i{\le}N$,
\begin{equation}\label{SDE}
  \diff L_i^N(t)=\sum_{j\in{\cal H}^N(i)} \int_{{\cal M}} Z_{ji}(u,L^N(t{-}))\,\overline{{\cal N}}_j(\diff t, \diff u)-L_i^N(t{-})\,{\cal N}_i(\diff t),
\end{equation}
where $f(t{-})$ denotes the left limit of the function $f$ at $t{>}0$. For  $i{\in}\{1,\ldots,N\}$,  the  points of the  process  ${\cal N}_i(\diff t)$ correspond to the instants when the $i$th urn is emptied. Recall the notation ${\cal N}_i(\diff t){=}\overline{\cal N}_i(\diff t, {\cal M})$. If time $t$ is one of these instants, due to the uniform  distribution assumption of  the variables $(U_k^\cdot)$, Relation~\eqref{Im} gives that, conditionally on ${\cal F}_{t-}$,  a ball from urn $i$ is allocated to urn $j$ with probability $p_{ij}^N(L^N(t-))$. This shows that the solution of Equation~\eqref{SDE} does represent our allocation process  of balls into the urns. 

\bigskip

\hrule

\medskip

\setcounter{assumption}{8}
\begin{assumption}[Initial State]\label{CondI} 
\begin{enumerate}
  {

  \item[]
    \item Invariance by translation.\\The  distribution of the vector $L^N(0){\steq{def}}(L_i^N(0))$ is such that
\begin{equation}\label{InvTrans}
\left(L_1^N(0),L_2^N(0),\ldots,L_N^N(0)\right)\steq{dist}\left(L_2^N(0),L_3^N(0),\ldots,L_N^N(0),L_1^N(0)\right)
\end{equation}
and  $L^N(0){\in}{\cal S}_N$, i.e.\  $L_1^N(0){+}L_2^N(0){+}\cdots{+}L_N^N(0)=F_N.$
\item[]
\item The local empirical distribution of the initial state  converges in distribution, for the total variation distance, to a random variable  $\Pi_0{\in}M_1(\N)$, i.e.\  to a random probability distribution on $\N$,
\begin{equation}\label{tvInit}
  \lim_{N\to+\infty} \E\left(\rule{0mm}{4mm}\|\Lambda_1^N(0){-} \Pi_0\|_{{\rm tv}}\right)=0,
\end{equation}
and such that  $\P(\croc{\Pi(0),I}{=}\beta){=}1$.
\item There exists some $\eta{>}0$ such that
\begin{equation}\label{expInit}
\sup_{N\in\N}\E\left(\int_\N e^{\eta x}\Lambda_1^N(\diff x)\right) =\sup_{N\in\N}\E\left(e^{\eta L_1^N(0)}\right)<{+}\infty. 
\end{equation}
    }
\end{enumerate}
\end{assumption}

\medskip

\hrule

\bigskip

{
  \noindent
Relation~\eqref{InvTrans} implies that, for $1{\le}i{\le}N$, $L_i^N(0)$ has the same distribution as $L_{i+1}^N(0)$, and therefore they are identically distributed. Definition~\eqref{Hsym} of the topology of the graph by translation of the set of neighboring nodes gives that  the local empirical distribution $\Lambda_i^N(0)$ at node $i$ has the same distribution as $\Lambda_1^N(0)$. The dynamics of the  evolution, see Equation~\eqref{SDE}, are also invariant (in distribution) under translation, consequently, for $t{\ge}0$, the variable $L_i^N(t)$ [resp. $\Lambda_i^N(t)]$ have also the same distribution as $L_1^N(t)$ [resp. $\Lambda_1^N(t)$].
}

\subsection*{Evolution Equations for Local Empirical Distributions}
Recall that, for any function $f$ with finite support on $\N$, 
\[
\croc{\Lambda^N_i(t),f}\steq{def} \int_\N f(x)\, \Lambda^N_i(t)(\diff x),
\]
The SDE~\eqref{SDE} can then be rewritten  in the following way,
\begin{multline}\label{SDE2}
\croc{\Lambda^N_i(t),f}=
\croc{\Lambda^N_i(0),f}+\frac{1}{h_N}\sum_{j{\in}{\cal H}^N(i)}\int_0^t \left[f(0){-}f(L_j^N(s{-}))\right]\,{\cal N}_j(\diff s)\\
{+}\frac{1}{h_N}\,\sum_{\mathclap{\substack{j{\in}{\cal A}^N(i)\\k\in{\cal H}^N(i)\cap{\cal H}^N(j)}}}\quad\int_{[0,t]\times{\cal M}}\left[f\left(\rule{0mm}{4mm}L_k^N(s{-}){+}Z_{jk}(u,L^N(s{-}))\right){-}f(L_k^N(s{-}))\right]\,\overline{\cal N}_j(\diff s, \diff u),
\end{multline}
where the variables $(Z_{jk}(\cdot,\cdot))$ are defined by Relation~\eqref{Zu} and ${\cal A}^N(i)$ by Relation~\eqref{range}.
\subsection{A Heuristic Asymptotic Description}
 We first present an informal, hopefully intuitive, motivation for the asymptotic SDE satisfied by the time evolution of the number of balls in a given urn. It should be noted that we will not establish our mean-field result in the same way. The method can be used in a simpler setting, see Sun et al.~\cite{SSMRS2}. It does not seem to be possible for our current model. 

We assume for this section  that the distribution of the $\Pi_0$ of Relation~\eqref{tvInit} is a Dirac measure at  $\pi_0{\in}M_1(\N)$. 
The integration of Equation~\eqref{SDE} and the use of Proposition~\ref{MPPMart} lead to the relation
\begin{equation}\label{eqL}
L_i^N(t)= L_i^N(0)+ C_i^N(t) -\int_0^t L_i^N(s{-})\,\overline{\cal N}_i(\diff s, {\cal M})
\end{equation}
 where $(C^N_i(t))$ is the process associated to the interaction of the nodes in the neighborhood of $i$,
   \[
C^N_i(t)= \sum_{j\in{\cal H}^N(i)} \int_{[0,t]\times{\cal M}} Z_{ji}(u,L^N(s{-}))\,\overline{{\cal N}}_j(\diff s, \diff u).
\]
As it will be seen, under Condition~\ref{CondA} of Section~\ref{subsecallocation},  with high probability the process $$(Z_{ji}(u,L^N(s{-})),u{\in}{\cal M},s{\le}t)$$ is either $0$ or $1$ and, consequently, $(C^N_i(t))$ is a counting process.  {This comes essentially from the fact  that, on a bounded time interval, when the balls of an urn are re-distributed,  the probability of having two balls assigned to the same urn of a given neighborhood will be of the order of $1/h_N$.}
Additionally,  the process 
$(C^N_i(t){-}\widehat{C}^N_i(t))$ is a martingale,
\[
\left(\widehat{C}^N_i(t)\right)=\left(\sum_{j\in{\cal H}^N(i)} \int_0^t L_j^N(s)p_{ji}^N(L^N(s)) \diff s\right)
\]
is the {\em compensator} of $(C_i^N(t))$.  See Example~4.3.4 p.~56 of Jacobsen~\cite{Jacobsen}. With the definition of $(p_{ji}^N(\ell))$, Relation~\eqref{PPsi} of Condition~\ref{CondA} gives the equivalence
\[
\left(h_Np_{ji}^N(L^N(t))\right)\sim \left(\Psi(\Lambda_j^N(t),L_i^N(t))\right) 
\]
with high probability on finite time intervals, then
\[
\left(\widehat{C}^N_i(t)\right)\sim \left(\int_0^t \frac{1}{h_N}\sum_{j\in{\cal H}^N(i)} L^N_j(s)  \Psi(\Lambda_j^N(t),L_i^N(s))\,\diff s\right).
\]
Assuming   that  the local empirical measures $(\Lambda^N_j(t))$ [resp. the processes $(L^N_j(t))$] are converging in distribution to a continuous deterministic process $(\Lambda(t))$ [resp. to a process $(L(t))$]. In particular $(\Lambda(t)(f)){=}(\E(f(L(t))))$ and, due to the fact that $L^N(t){\in}{\cal S}_N$ and the scaling assumption~\eqref{Scaling},
\[
\croc{\Lambda(t),I}{=}\E(L(t)){=}\beta,\quad \forall t{\ge}0,
\]
where $I(x){=}x$ is the identity. Under this hypothesis, one would have the equivalence in distribution for the compensator of $(C^N_i(t))$,
\begin{multline*}
\left(\widehat{C}^N_i(t)\right)\sim \left(\int_0^t  \Psi\left(\Lambda(s),L^N_i(s)\right) \frac{1}{h_N}\sum_{j\in{\cal H}^N(i)} L^N_j(s) \,\diff s.\right)
\\\sim \left(\widehat{C}(t)\right)\steq{def}\left(\beta\int_0^t\Psi(\Lambda(s),L(s))\,\diff s\right).
\end{multline*}
This suggest that the sequence of counting processes $(C_i^N(t))$ is converging in distribution to a counting process $(C(t))$ given by
\[
(C(t))=\left(\int_0^t \overline{\cal P}\left(\rule{0mm}{4mm}\diff s{\times}[0,\beta \Psi(\Lambda(s),L(s{-}))]\right)\right),
\]
where $\overline{\cal P}$ is an homogeneous Poisson point process on $\R_+^2$.

In view of Relation~\eqref{eqL}, a possible limit $(L_1^N(t))$ of $(L_1^N(t))$ when $N$ is large should satisfy the following SDE
\[
\diff L(t)= \overline{\cal P}\left(\rule{0mm}{4mm}\diff t{\times}[0,\beta \Psi(\Lambda(t),L(t{-}))]\right) - \overline{\cal N}_1(\diff t, {\cal M}).
\]
 We first establish the existence and uniqueness of such a process. The proof of the convergence in distribution of $(L^N_i(t))$, $1{\le}i{\le}N$, to this asymptotic process is achieved in the next section.
\subsection{The McKean{-}Vlasov process}

\begin{theorem}\label{theoMcKV}
If the functional $\Psi$ satisfies Condition~\ref{CondA}, there exists a unique c\`adl\`ag process $(L(t))$ with an initial condition $L(0){\steq{dist}}\pi(0){\in}M_1(\N)$ and such that  the SDE
\begin{equation}\label{McKV}
\diff L(t) =\overline{\cal P}\left(\diff t{\times}\left[0,\beta\,\rule{0mm}{4mm} \Psi(\pi(t),L(t{-}))\right]\right){-}L(t{-}) {\cal P}(\diff t)
\end{equation}
holds, where, for $t{>}0$, $\pi(t)$ is the distribution of $L(t)$ on $\N$  and  $\overline{\cal P}$ [resp. ${\cal P}$]  is an homogeneous Poisson point process on $\R_+^2$ [resp. $\R_+$] with rate $1$ and the random processes $\overline{{\cal P}}$ and ${\cal P}$ are independent .
\end{theorem}

This will be referred to as the McKean{-}Vlasov process associated to this model. The associated process  $(\pi(t))$ is  a continuous $M_1(\N)$-valued function, for any function $f$ with finite support on $\N$, $\croc{\pi(t),f}=\E(f(L(t)))$ and the equation 
\[
\croc{\pi(t),f}=\croc{\pi(0),f}+\int_0^t \croc{\pi(s),\Omega_{\pi(s)}f}\,\diff s, \quad \forall t{\ge}0,
\]
holds where, for $\sigma{\in}M_1(\N)$,  the generator $\Omega_{\sigma}$ is defined by
\begin{equation}\label{Omega}
\Omega_\sigma(f)(x)=\beta\Psi(\sigma,x)\left(f(x{+}1){-}f(x)\right)+ (f(0){-}f(x)),
\end{equation}
for $x{\in}\N$. 

{
To prove the mean-field convergence under Condition~\ref{CondI}, it is convenient to introduce the random process on $M_1(N)$ satisfying the integral equation
\begin{equation}\label{FP}
\croc{\Pi(t),f}=\croc{\Pi(0),f}+\int_0^t \croc{\Pi(s),\Omega_{\Pi(s)}f}\,\diff s, \quad \forall t{\ge}0.
\end{equation}
Recall that $\Pi(0)$ is the asymptotic local empirical distribution of a node at time $0$, it is a priori random. Its distribution is therefore an element of $M_1(M_1(\N))$.

\begin{proof}[Proof of Theorem~\ref{theoMcKV}]
  The result is a direct consequence of Theorem~2.1 of Graham~\cite{Graham2}. All we have to prove is that there exists a constant $C_0$ such that, for all $\sigma$, $\sigma'{\in}M_1(\N)$ and $\ell$, $\ell'{\in}\N$, the relation
  \[
  \int_0^{+\infty} \left|\ind{u{\le}\Psi(\sigma,\ell)}-\ind{u{\le}\Psi(\sigma',\ell')}\right|\,\diff u \le C_0\left(\rule{0mm}{4mm}|\ell{-}\ell'|{+}W_1(\sigma,\sigma')\right)
  \]
holds, where $W_1$ is the Wasserstein distance on $M_1(\N)$ defined by Relation~\eqref{Wp}. 
  
From the Lipschitz condition~\eqref{LipPsi} of Assumption~\ref{CondA}, we obtain that
\begin{multline*}
    \int_0^{+\infty} \left|\ind{u{\le}\Psi(\sigma,\ell)}-\ind{u{\le}\Psi(\sigma',\ell')}\right|\,\diff u \\ =
    \left|\Psi(\sigma,\ell){-}\Psi(\sigma',\ell')\right|  \leq \max(C_\Psi,D_\psi)\left(\rule{0mm}{4mm} |\ell{-}\ell'|{+}\|\sigma{-}\sigma'\|_{{\rm tv}}\right).
\end{multline*}
The elementary inequality
$
\ind{x{\not=}y}\leq |x{-}y|
$,
when $x$ and $y$ are integers, and Relation~\eqref{TV} give the relation $\|\sigma{-}\sigma'\|_{{\rm tv}}{\le}W_1(\sigma,\sigma')$. The theorem is proved.
\end{proof}
The next result is a simple, but important, invariance result for the McKean{-}Vlasov process. 
\begin{prop}[Conservation of Mass]\label{remmass}
With the same notations as in Theorem~\ref{theoMcKV}, if the variable $L(0)$ is integrable with $\E(L(0)){=}\beta{>}0$, then $\E(L(t)){=}\beta$ for all $t{\ge}0$. 
\end{prop}

\noindent
This result can also be expressed as $\croc{\pi(t),I}{=}\beta$ for all $t{\ge}0$. Recall the $I$ is the identity function on $\N$, $I(x){=}x$ for $x{\in}\N$. 
\begin{proof}
By integrating Equation~\eqref{McKV} and by using the fact that $\Psi$ is bounded by Assumption~\ref{CondA}, we get that $L(t)$ is integrable for all $t{\ge}0$ and 
  \[
  E(L(t)){-}\beta=\beta\int_0^t \E(\Psi(\pi(s),L(s)))\,\diff s-\int_0^t\E(L(s))\,\diff s.
  \]
Relation~\eqref{Psimass} of  Assumption~\ref{CondA} gives that $\E(\Psi(\pi(s),L(s))){=}1$, for all $s{\ge}0$.  It is then easy to conclude that $(\E(L(t)))$ is constant equal to $\beta$. 
\end{proof}
In particular, under Assumption~\ref{CondI}, the solution $(\Pi(t))$  of Equation~\eqref{FP} is such that, almost surely, $\croc{\Pi(t),I}{=}\beta$ for all $t{\ge}0$. 
}

{We conclude the section with two technical results which will be used in the proof of the  mean-field convergence. The first one gives the existence of an exponential moment for the  McKean{-}Vlasov process, this is in fact a key ingredient in the proof of Theorem~\ref{theoLED} for the convergence in distribution of the local empirical distributions. 
\begin{prop}\label{propExp}
If the functional $\Psi$ satisfies Condition~\ref{CondA} of Section~\ref{subsecallocation} and   there exists some $\eta{>}0$ such that
  \[
  \E\left(\int_\N e^{\eta x}\Pi(0)(\diff x)\right){<}{+}\infty,
  \]
then, for  $T{>}0$, 
  \[
 \E\left(  \sup_{t{\le}T} \int_\N e^{\eta x}\Pi(t)(\diff x)\right)<{+}\infty,
  \]
  where $(\Pi(t))$ is the solution of the integral equation~\eqref{FP}.
\end{prop}
\begin{proof}
With a convenient probability space, for $\pi{\in}M_1(\N)$,  we denote by $(L_\pi(t))$ the solution of the SDE~\eqref{McKV} with initial distribution $\pi$. 
From the SDE, we get  the following inequality, for all $t{\le}T$,
\[
    \sup_{t\le T}  L_\pi(t)\leq L_\pi(0){+}\overline{\cal P}\left([0,T]{\times}[0,\beta\|\Psi\|_{\infty}]\right)\\
\]
By applying this estimate we get, for $\eta{>}0$, almost surely, for the distribution of $\Pi(0)$ in $M_1(M_1(\N))$, for $\pi{\in}M_1(\N)$,
\begin{multline*}  
  \left. \E\left( \sup_{t{\le}T} \int_\N e^{\eta x}\Pi(t)(\diff x)\right|\Pi(0){=}\pi\right)
  =\sup_{t{\le}T} \E\left(e^{\eta L_\pi(t)}\right)\\
  \leq e^{\eta L_{\pi}(0)} \E\left(e^{\eta\overline{\cal P}\left([0,T]{\times}[0,\beta\|\Psi\|_{\infty}]\right)}\right)
  {=}\int_\N e^{\eta x}\Pi(0)(\diff x)\E\left(e^{\eta\overline{\cal P}\left([0,T]{\times}[0,\beta\|\Psi\|_{\infty}]\right)}\right)
\end{multline*}
We conclude by integrating the inequality with respect to the distribution of $\Pi(0)$. 
\end{proof}
}
{The following simple lemma is used in the proof of Proposition~\ref{propI} to derive a Gr\"onwall-like inequality for the distance between the empirical distribution and the McKean{-}Vlasov process.}
\begin{lemma}\label{LemOm}
If $\Omega_.$ is the operator defined by Relation~\eqref{Omega}, then under Assumption~\ref{CondA} of Section~\ref{subsecallocation},   for any function $f$ on $\N$ with $\|\nabla_1(f)\|_{\infty}{<}{+}\infty$,   the relation
  \[
  \left|\croc{\sigma,\Omega_{\sigma}(f)}{-}\croc{\sigma',\Omega_{\sigma'}(f)}\right|\leq 2\beta\|\nabla_1(f)\|_{\infty} \left[\|\Psi\|_\infty{+}D_\Psi\right]\|\sigma{-}\sigma'\|_{{\rm tv}}{+}|\croc{\sigma{-}\sigma',f}|
  \]
  holds for any $\sigma$, $\sigma'{\in}M_1(\N)$ such  that $f$ is integrable with respect to $\sigma$ and $\sigma'$. 
\end{lemma}
\begin{proof}
  This is a simple consequence of the inequality
\[
\left|\croc{\sigma,\Omega_{\sigma}(f)}{-}\croc{\sigma',\Omega_{\sigma'}(f)}\right|
\leq | \croc{\sigma,\Omega_{\sigma}(f){-}\Omega_{\sigma'}(f)}|+
| \croc{\sigma{-}\sigma',\Omega_{\sigma'}(f)}|
\]
and Relation~\eqref{LipPsi}.
\end{proof}

\section{Mean-Field Convergence Results}\label{MFsec}
{
We establish the main convergence results, namely that under Assumptions~\ref{CondT}, \ref{CondA} and~\ref{CondI},  the process of the local empirical distribution at a node is converging in distribution to the solution of the Fokker-Planck equation~\eqref{FP}.  It is also shown that the same result holds for the (global) empirical distribution.  The important result of this section is the following theorem. }
\begin{theorem}[Convergence of Local Empirical Distributions]\label{theoLED}
If Assumptions~\ref{CondT}, \ref{CondA} and~\ref{CondI} hold, then, for any $1{\le}i{\le}N$, the local empirical distribution process at node~$i$, $(\Lambda^N_i(t))$,
  \[
(\Lambda^N_i(t))=\left(\frac{1}{\card({\cal H}^N(i))}\sum_{j\in{\cal H}^N(i)} \delta_{L^N_j(t)}\right)
  \]
  is converging to $(\Pi(t))$, the solution of the integral equation~\eqref{FP}, in the following sense
\begin{equation}\label{CVLE}
  \lim_{N\to\infty}\E\left(\sup_{s\le t}\|\Lambda_i^N(s){-}\Pi(s)\|_{{\rm tv}}+\sup_{s\le t}|\langle\Lambda_i^N(s){-}\Pi(s),I\rangle|\right)=0.
\end{equation}
In particular the process $(\Lambda^N_i(t))$ converges to $(\Pi(t))$ for the Wasserstein metric $\overline{{\cal W}}_T$ defined by Relation~\eqref{Wass2}.
\end{theorem}
{
  The statement of Assumption~\ref{CondT} is in Section~\ref{subsecgraph}, that of~\ref{CondA}   is in Section~\ref{subsecallocation} and for Assumption~\ref{CondI}  it is in Section~\ref{SecEvol}.

\medskip

The proof of this result is quite technical. 
The  second term within the expected value of Relation~\eqref{CVLE} is  related to a uniform $L_1$-convergence on finite time intervals  of the  local mean on the neighboring nodes of node $i$.  This term plays an important role in the convergence result as it will be seen. When $\Pi(0)$ of Assumption~\ref{CondI} is deterministic,  the process $(\Pi(t))$ is deterministic and from Proposition~\ref{remmass} of Section~\ref{SecEvol},  we have $\langle\Pi(t),I\rangle{=}\beta$, for all $t{\ge}0$. The theorem shows in particular that the limit of the process of the local average of the load in a given neighborhood is constant and equal to $\beta$.  The total average, i.e.\ when all nodes are considered, is deterministic and equal to $F_N/N$, its convergence is just the scaling assumption~\eqref{Scaling}. 

The general strategy to establish the mean-field convergence  is to use the system of evolution equations~\eqref{SDE2} and decompose it in a convenient way with the help of stochastic calculus for Poisson process and with various  estimates. This is done by proving successively technical results: Lemma~\ref{lem2} for the boundedness of the second moments of the number of balls in an urn, Lemma~\ref{lemC} to show that an urn receives at most one ball at each event on any finite time interval, and Lemma~\ref{lemM} to prove that the martingales vanish in the limit.  A Gr\"onwall's Inequality for the distance to the McKean{-}Vlasov process is established with  Propositions~\ref{propAB} and~\ref{propI}.  The theorem is then proved. The mean-field convergence, i.e.\ the convergence in distribution of the empirical distribution, is established in Theorem~\ref{theoED}. }

\medskip

We begin by recalling and  introducing some notations which will be used throughout this section.
\begin{itemize}
\item As before, if $f$ is some function on $\N$, for $y{\in}\N$, $\nabla_y$ is the {discrete gradient} operator defined by,  for $x{\in}\N$, $$\nabla_y(f)(x){\steq{def}}f(x{+}y){-}f(x).$$
\item The set of $1${-}Lipschitz functions on $N$ is denoted as ${\rm Lip}(1)$,
  \[
  {\rm Lip}(1)=\{f{:}\N{\to}\R: \|\nabla_1(f)\|_\infty{\le}1\},
  \]
  and, for $K{\ge}0$,
  \begin{equation}\label{IIK}
    I_K(x){\steq{def}}x\ind{x{\ge}K} \quad \text{\rm  and,  as before, } I(x){\steq{def}}x.
  \end{equation}
\end{itemize}
Relation~\eqref{SDE2} is rewritten by compensating the stochastic integrals with respect to the Poisson processes then, by using Proposition~\ref{MPPMart}, one gets that, for a bounded function $f$  on $\N$ and $t{\ge}0$,
\begin{multline}\label{SDE3}
\croc{\Lambda^N_i(t),f}=
\croc{\Lambda^N_i(0),f}+\int_0^t \croc{\Lambda^N_i(s),f(0){-}f(\cdot)} \diff s
\\{+} \sum_{z\ge 1}  \int_0^t\sum_{j{\in}{\cal H}^N(i)}\frac{1}{h_N}\quad\nabla_{z}(f)\left(L_j^N(s)\right)\left(\sum_{k{\in}{\cal H}^N(j)}B_{kj}(L^N(s))(\diff z)\right) \diff s+M_{f,i}^N(t),
\end{multline}
$B_{kj}(\cdot)$ is the binomial distribution defined by Relation~\eqref{Bij} and $(M_{f,i}^N(t))$ is a martingale whose previsible increasing process $(\langle M_{f,i}^N(t)\rangle)$ is given, via some simple calculations, by
\begin{multline}\label{crocM}
\croc{M_{f,i}^N(t)}=
\frac{1}{h_N^2}\sum_{j{\in}{\cal A}^N(i)} \int_{0}^t\int_{(z_i)\in\N^N}\left(\rule{0mm}{7mm}\left[f(0){-}f(L_j^N(s))\right]\ind{j\in{\cal H}^N(i)}\right.\\\left. {+}\sum_{{k{\in}{\cal H}^N(i)\cap{\cal H}^N(j)}} \nabla_{z_k}(f)(L_k^N(s))  \right)^2B_j(L^N(s))(\diff z_1,\ldots, \diff z_N)\diff s.
\end{multline}
where ${\cal A}^N(i)$ is defined by Relation~\eqref{range} and $B_{j}(\cdot)$ is the multinomial distribution defined by Relation~\eqref{Mij}.

We introduce the potential asymptotic process of Theorem~\ref{theoMcKV}  in this relation. A careful  (and somewhat cumbersome) rewriting of Relation~\eqref{SDE3} gives the identity
\begin{multline}\label{SDE4}
\croc{\Lambda^N_i(t),f}=
\croc{\Lambda^N_i(0),f}+\int_0^t \croc{\Lambda^N_i(s),\Omega_{\Lambda^N_i(s)}(f)} \diff s\\
+X_{f,i}^N(t)+Y_{f,i}^N(t)+Z_{f,i}^N(t)+M_{f,i}^N(t),
\end{multline}
where $\Omega_.$ is the operator defined by Relation~\eqref{Omega} and $I$ is the identity function. The other terms are
\begin{multline}\label{eqA}
X_{f,i}^N(t){\steq{def}}\\ \frac{1}{h_N}\int_0^t\quad \sum_{\mathclap{j\in{\cal H}^N(i)}} \nabla_1(f)(L_j^N(s))\croc{\Lambda_{j}^N(s){-}\Pi(s),I} \Psi\left(\Lambda_i^N(s),L_j^N(s)\right)\,\diff s,
\end{multline}
where $(\Pi(s))$ is defined by Equation~\eqref{FP},
\begin{multline}\label{eqB}
  Y_{f,i}^N(t)\steq{def}\\
  \frac{1}{h_N}\int_0^t\;\sum_{\mathclap{\substack{j\in{\cal H}^N(i)\\k\in{\cal H}^N(j)}}} \nabla_1(f)(L_j^N(s))\frac{L_k^N(s)}{h_N}\left(\rule{0mm}{4mm}\Psi\left(\Lambda_k^N(s),L_j^N(s)\right){-}\Psi\left(\Lambda_i^N(s),L_j^N(s)\right)\right)\diff s.
\end{multline}
 Note that the almost sure relation $\croc{\Pi(s),I}{=}\beta$, $s{\ge}0$ has been  used in this derivation.
The term $L_k^N(s)/h_N$ in the expression of $(  Y_{f,i}^N(t))$ is the main source of difficulty to prove the mean-field  convergence. When the sequence $(h_N)$ grows linearly with $N$ then, since $|L_k^N|{\leq}F_N{\sim}\beta N$, this term is bounded and the usual contraction methods, via Gr\"onwall's Inequality, can be used without too much difficulty.  A more careful approach has to be considered if the growth of $(h_N)$ is sublinear. 
Finally, 
\begin{multline}\label{eqC}
  Z_{f,i}^N(t){\steq{def}}
  \int_0^t\sum_{j{\in}{\cal H}^N(i)}\hspace{-3mm}\frac{1}{h_N}\sum_{z{\ge}2}\nabla_{z}(f)\left(L_j^N(s)\right)\left(\sum_{k{\in}{\cal H}^N(j)}B_{kj}(L^N(s))(\diff z)\right) \diff s\\{+}
 \int_0^t\sum_{\mathclap{\substack{j{\in}{\cal H}^N(i)\\k{\in}{\cal H}^N(j)}}}\frac{1}{h_N}\nabla_{1}(f)\left(L_j^N(s)\right)
 L^N_k(s)p^N_{kj}\left(L^N(s)\right)\left(\left(1{-}p^N_{kj}\left(L^N(s)\right)\right)^{L^N_k(s)-1}\hspace{-3mm}{-}1\right) \diff s\\
+ \int_0^t\sum_{\mathclap{\substack{j{\in}{\cal H}^N(i)\\k{\in}{\cal H}^N(j)}}}\frac{1}{h_N}\quad\nabla_{1}(f)\left(L_j^N(s)\right)\frac{L^N_k(s)}{h_N}\left(h_Np^N_{kj}\left(L^N(s)\right){-}\Psi\left(\Lambda^N_k(s),L^N_j(s)\right)\right)\diff s
\end{multline}
where $(B_{kj}(\ell))$ are the binomial distributions  defined by Relation~\eqref{Bij}.

We first consider the last four  terms of Relation~\eqref{SDE4} via  three  technical lemmas. 
\begin{lemma}\label{lem2}
Under Assumptions~\ref{CondA} and~\ref{CondI} of Section~\ref{SecEvol}, for any $T{\ge}0$, there exists a finite constant $C_T$ such that
  \[
  \sup_{N\ge 1}\sup_{t\le T} \E\left(L^N_1(t)^2\right) \leq C_T.
  \]
\end{lemma}
\begin{proof}
Assumption~\ref{CondI} shows that the quantity
    \[
 q_0{\steq{def}} \sup_{N\geq 1}\E\left(L^N_1(0)^2\right)
 \]
 is finite.  From  Relation~\eqref{PPsi} and the boundedness of the functional $\Psi$, we get  the existence of a constant $C_0$ and $N_0{\in}\N$, such that, for $N{\ge}N_0$, the relation
  \begin{equation}\label{Binom+}
    \sup_{i\in{\cal H}^N,\ell\in{\cal S}_N} p^N_i(\ell)\leq \frac{C_0}{h_N}<1
  \end{equation}
holds. Proposition~\ref{MPPMart} and Relation~\eqref{SDE} give the identity
\begin{multline}\label{eqa1}
  \E\left(L^N_1(t)^2\right){=}  \E\left(L^N_1(0)^2\right)\\{+} \sum_{\mathclap{j{\in}{\cal H}^N}}\int_0^{\mathclap{t}}\E\left(\int_\N\left(2L^N_1(s)z{+}z^2\right)B_{j1(L^N(s))}(\diff z)\right)\diff s
  {-}\int_0^{\mathclap{t}}   \E\left(L^N_1(s)^2\right)\diff s.
\end{multline}
For $s{\ge}0$, by using Relation~\eqref{Binom+} and the symmetry of the model, we get 
\begin{multline*}
  \E\left(L^N_1(s) \int_\N z B_{j1(L^N(s))}(\diff z)\right)=\E\left(\rule{0mm}{4mm}L^N_1(s)L^N_j(s)p^N_{j1}\left(L^N(s)\right) \right)\\
  \leq \frac{C_0}{h_N}\E\left(\rule{0mm}{4mm}L^N_1(s)L^N_j(s) \right)   \leq \frac{C_0}{h_N}\E\left(\rule{0mm}{4mm}L^N_1(s)^2\right),
\end{multline*}
by Cauchy-Shwartz's Inequality. Similarly, by using the expression of the second moment of a binomial variable, 
\begin{multline*}
  \E\left(\int_\N z^2 B_{j1(L^N(s))}(\diff z)\right)\leq \E\left(\rule{0mm}{4mm}L^N_1(s)^2p^N_{j1}\left(L^N(s)\right)^2 \right){+}\E\left(\rule{0mm}{4mm}L^N_1(s)p^N_{j1}\left(L^N(s)\right) \right)\\
  \leq \frac{2C_0}{h_N}\E\left(\rule{0mm}{4mm}L^N_1(s)^2\right).
\end{multline*}
By plugging these estimates in Equation~\eqref{eqa1},  we obtain the following inequality, for all $N{\ge}1$,
\[
  \E\left(L^N_1(t)^2\right)\leq  q_0{+}(4C_0{-}1)\int_0^t\E\left(L^N_1(s)^2\right)\,\diff s.
  \]
A straightforward use of Gr\"onwall's Inequality gives then directly the estimation since the constants $q_0$ and $C_0$ do not depend on $N$. The lemma is proved. 
\end{proof}
\begin{lemma}\label{lemC}
Under Assumptions~\ref{CondT} of Section~\ref{subsecgraph}, \ref{CondA} of Section~\ref{subsecallocation} and~\ref{CondI} of Section~\ref{SecEvol}, if $(Z_{f,i}^N(t))$ is the process defined by Relation~\eqref{eqC},  for $T{\ge}0$,
  \[
  \|Z^N\|_T{\steq{def}}\E\left(\sup_{t\leq T}\sup_{f\in{\rm Lip}(1)} \left| Z_{f,i}^N(t)\right|\right)
  \]
  then the sequence $(\|Z^N\|_T)$ converges to $0$. 
\end{lemma}
\begin{proof}
  We denote by $\delta^N_1(f,t)$, $\delta^N_2(f,t)$ and $\delta^N_3(f,t)$ the three terms of the right hand side of Definition~\eqref{eqC} of $(Z_{f,i}^N(t))$.

  Let, for $l{\in}\N$, $\overline{B}(l)$ be binomial distribution with parameter $l$ and $p_0{\steq{def}}C_0/h_N$ defined in Relation~\eqref{Binom+},  then
\begin{equation}\label{BinIneq}
 \E\left(\overline{B}(l)\ind{\overline{B}(l)\ge 2}\right)=\E\left(\overline{B}(l)\right){-}\P\left(\overline{B}(l){=}1\right)
  {=}l p_0\left(1{-}(1{-}p_0)^{l-1}\right)\leq (l p_0)^2.
\end{equation}
By using this inequality and the fact that if $f{\in}{\rm Lip}(1)$,  then $\|\nabla_z(f)\|_\infty{\leq} z$ holds, for $z{\ge}1$ .  For $T{>}0$,
\begin{multline*}
  \E\left(\sup_{t\leq T} \sup_{f\in{\rm Lip}(1)} \left|\delta^N_1(f,t)\right|\right)\\\leq \frac{1}{h_N}\int_0^T \sum_{\substack{j\in{\cal H}^N(i)\\k\in{\cal H}^N(j)}}\E\left(\overline{B}(L^N_k(s))\ind{\overline{B}(L^N_k(s))\ge 2}\right)\diff s
  \leq  \frac{C_0^2}{2h_N}\int_0^T \E\left(L^N_k(s)^2\right)\diff s.
\end{multline*}
From Lemma~\ref{lem2}  we deduce that the right hand side of the relation is converging to $0$ as $N$ gets large.  A similar argument can also be used for the term  $(\delta^N_2(f,t))$. For the last term of $(Z_{f,i}^N(t))$
\begin{multline*}
\E\left( \sup_{t\leq T}\sup_{f\in{\rm Lip}(1)}\left|\delta^N_3(f,t)\right|\right)\\\leq
\sup_{\substack{i\in{\cal H}^N\\\ell\in{\cal S}_N}}   \left| h_Np_i^N(\ell){-}\Psi\left(\frac{1}{h_N}\sum_{j{\in}{\cal H}^N} \delta_{\ell_j},\ell_i\right)\right|
\int_0^T \E\left(L^N_k(s)\right)\diff s.
\end{multline*}
Relation~\eqref{PPsi} of Assumption~\ref{CondA} shows that this term is converging to $0$ when $N$ goes to infinity. The lemma is proved. 
\end{proof}
\begin{lemma}\label{lemM}
Under Assumptions~\ref{CondT} of Section~\ref{subsecgraph} and~\ref{CondI} of Section~\ref{SecEvol}, the relation
  \[
\lim_{N\to+\infty} \sup_{f\in{\rm Lip}(1)} \E\left(\sup_{t\leq T}\left(M_{f,i}^N(t)\right)^2\right)=0,
\]
holds for  any $T{\ge}0$,  where, for $f{\in}{\rm Lip}(1)$, $(M_{f,i}^N(t))$ is  the martingale  of  SDE~\eqref{SDE3}.
\end{lemma}
\begin{proof}
By using Relation~\eqref{crocM}, we get  the inequality
  \begin{multline*}
     \E\left(\sup_{f\in{\rm Lip}(1)} \croc{M_{f,i}^N}(T)\right)\leq     \frac{2}{h_N^2}\sum_{j{\in}{\cal A}^N(i)}  \int_0^T\left[\rule{0mm}{8mm}\E\left(L^N_j(s)^2\right)\right.\\\left. {+}\E\left(\int_{(z_i)\in\N}\left(\sum_{{k{\in}{\cal H}^N(i)}} z_k\right)^2 B_{j}(L^N(s))(\diff z_1,\ldots,\diff z_N)  \right)\right]\diff s.
  \end{multline*}
Note that, for $\ell{\in}{\cal S}_N$, the multinomial distribution $B_j(\ell)$ on $\N^N$  has the support $\{z{\in}\N^N{:} z_1{+}z_2{+}\cdots{+}z_N{=}\ell_j\}$, which gives the relation
\[
     \E\left(\sup_{f\in{\rm Lip}(1)} \croc{M_{f,i}^N}(T)\right)\\\leq 
    4\frac{\card({\cal A}^N(i))}{h_N^2}  \int_0^T\E\left(L^N_1(s)^2\right)\diff s.
\]
  We conclude the proof by using again Lemma~\ref{lem2}, Assumption~\ref{CondT} and Doob's Inequality. 
    
\end{proof}
Now we can turn to the remaining terms of Relation~\eqref{SDE4} to establish the  convergence results. 
\begin{prop}\label{propAB}
Under Assumption~\ref{CondA} of Section~\ref{subsecallocation}, for $T{\ge}0$ and $t{\le}T$,
  \begin{equation}\label{eqb2}
     \|X^N\|_t{\steq{def}}\E\left(\sup_{f\in{\rm Lip}(1)} \sup_{s\leq t}\left|X_{f,1}^N(s)\right|\right)\leq \|\Psi\|_{\infty}\int_0^t \E\left(\left|\croc{\Lambda^N_1(s){-}\Pi(s),I}\right|\right)\,\diff s,\\
  \end{equation}
  and for $K{>}0$,
  \begin{multline}\label{eqb3}
     \|Y^N\|_t{\steq{def}}\E\left(\sup_{f\in{\rm Lip}(1)} \sup_{s\leq t}\left|Y_{f,1}^N(s)\right|\right)\leq 4D_\Psi K\int_0^{\mathclap t} \E\left(\left\|\Lambda_1^N(s){-}\Pi(s)\right\|_{\rm tv}\right)\,\diff s\\+
    D_\Psi\int_0^t \E\left(\left|\croc{\Lambda^N_1(s){-}\Pi(s),I}\right|\right)\,\diff s+ D_\Psi\int_0^t \E\left(\croc{\Pi(s),I_K}\right)\,\diff s
  \end{multline}
  where  $I$ is the identity function, $(X_{f,1}^N(t))$ and $(Y_{f,1}^N(t))$ are defined respectively by Relations~\eqref{eqA} and~\eqref{eqB},  and $(\Pi(t))$ by Equation~\eqref{FP}.
\end{prop}
\begin{proof}
  The first Inequality is straightforward to derive from Relation~\eqref{eqA}.

  Let $f{\in}{\rm Lip}(1)$ and $t{\le}T$, by the Lipschitz property of Relation~\eqref{LipPsi} we get that
  \begin{align}
\sup_{s\le t}\left|Y_{f,1}^N(s)\right|&\leq    \frac{1}{h_N^2}\int_0^t\;\sum_{\mathclap{\substack{j\in{\cal H}^N(1)\\k\in{\cal H}^N(j)}}} L_k^N(s)\left|\rule{0mm}{4mm}\Psi\left(\Lambda_k^N(s),L_j^N(s)\right){-}\Psi\left(\Lambda_1^N(s),L_j^N(s)\right)\right|\diff s\label{eqb1}\\
    &\leq    \frac{D_\Psi}{h_N^2} \int_0^t\sum_{j\in{\cal H}^N(1)}\sum_{k\in{\cal H}^N(j)}  L_k^N(s)\left\|\Lambda_k^N(s){-}\Lambda_1^N(s)\right\|_{\rm tv}\diff s. \notag
  \end{align}
For $s{\ge}0$, 
\begin{multline*}
  \frac{1}{h_N^2}\sum_{j\in{\cal H}^N(1)}\sum_{k\in{\cal H}^N(j)} L_k^N(s)\left\|\Lambda_k^N(s){-}\Lambda_1^N(s)\right\|_{\rm tv}\\
  \leq    \frac{1}{h_N^2} \sum_{j\in{\cal H}^N(1)}\sum_{k\in{\cal H}^N(j)} \left[ K\left\|\Lambda_k^N(s){-}\Lambda_1^N(s)\right\|_{\rm tv}+L_k^N(s)\ind{L_k^N(s){\ge} K}\right],
\end{multline*}
since $\|\Lambda_k^N(s){-}\Lambda_1^N(s)\|_{\rm tv}{\le}1$. We get therefore, by symmetry, that
\begin{align*}
 \frac{1}{h_N^2}&\sum_{j\in{\cal H}^N(1)}\sum_{k\in{\cal H}^N(j)} \E\left( L_k^N(s)\left\|\Lambda_k^N(s){-}\Lambda_1^N(s)\right\|_{\rm tv}\right)\\
&  \leq    \frac{1}{h_N^2} \sum_{j\in{\cal H}^N(1)}\sum_{k\in{\cal H}^N(j)} \left[2K\E\left(\left\|\Lambda_1^N(s){-}\Pi(s)\right\|_{\rm tv}\right)+\E\left(L_1^N(s)\ind{L_1^N(s){\ge} K}\right)\right]\\
& \leq 2K\E\left(\left\|\Lambda_1^N(s){-}\Pi(s)\right\|_{\rm tv}\right)+\E\left(L_1^N(s)\ind{L_1^N(s){\ge} K}\right),
\end{align*}
and this  term is smaller than
\begin{multline*}
  2K\E\left(\left\|\Lambda_1^N(s){-}\Pi(s)\right\|_{\rm tv}\right){+} \E\left(\left|\croc{\Lambda_1^N(s){-}\Pi(s),I_K}\right|\right){+}\croc{\Pi(s),I_K}\\
  \leq   4K\E\left(\left\|\Lambda_1^N(s){-}\Pi(s)\right\|_{\rm tv}\right){+}\E\left(\left|\croc{\Lambda_1^N(s){-}\Pi(s),I}\right|\right){+}\croc{\Pi(s),I_K},
\end{multline*}
which gives  the desired result. 
\end{proof}
For the next step to prove the main mean-field result, one has to estimate the deviations of the local mean,
\[
 \E\left(\sup_{s\leq t}\left|\croc{\Lambda^N_1(s){-}\Pi(s),I}\right|\right){=} \E\left(\sup_{s\leq t}\left|\croc{\Lambda^N_1(s),I}{-}\beta\right|\right),
 \]
this is the next proposition.   We define, for $t{\ge}0$, $  d^N(t){\steq{def}} d_1^N(t){+}  d_2^N(t)$, with
 \begin{equation}\label{defd}
\begin{cases}
d_1^N(t)\steq{def} &  \displaystyle\E\left(\sup_{s\le t}\left\|\Lambda_1^N(s){-}\Pi(s)\right\|_{\rm tv}\right),\\
 d_2^N(t)\steq{def} &  \displaystyle\E\left(\sup_{s\leq t}\left|\croc{\Lambda^N_1(s){-}\Pi(s),I}\right|\right).
\end{cases}
 \end{equation}
 We are going to show  that, for $T{>}0$, the sequence $(d^N(T))$  is converging to $0$. Let
 \[
 \overline{d}(t)=\limsup_{N\to+\infty} d^N(t).
 \]
\begin{prop}\label{propI}
Under Assumptions~\ref{CondT}, \ref{CondA} and~\ref{CondI},  for any $T{>}0$, there exists a constant $C_0{>}0$ such that the relation
  \begin{equation}\label{IneqI}
\overline{d}(t)\leq  C_0 K\int_0^t \overline{d}(s)\,\diff s{+}C_0T K\E\left(\sup_{s\leq T} \croc{\Pi(s),I_K}\right).
  \end{equation}
  holds for all $t{\le}T$ and $K{\ge}1$. 
\end{prop}
\begin{proof}
Noting that $I{\in}{\rm Lip}(1)$, Relations~\eqref{FP} and~\eqref{SDE4} give the inequality, for $T{>}0$ and $t{\le}T$,
  \begin{multline*} \E\left(\sup_{s\leq t}\left|\croc{\Lambda^N_1(s){-}\Pi(s),I}\right|\right)\leq
        \E\left(\left|\croc{\Lambda^N_1(0){-}\Pi(0),I}\right|\right)\\{+}
        \int_0^t\hspace{-2mm} \E\left(\left|\croc{\Lambda^N_1(s),\Omega_{\Lambda^N_1(s)}(I)}{-}\croc{\Pi(s),\Omega_{\Pi(s)}(I)}\right|\right) \diff s\\
        {+}\E\left(\sup_{s\le T}\left|M_{I,1}^N(s)\right|\right){+}\left\|X^N\right\|_t{+}\left\|Y^N\right\|_t{+}\left\|Z^N\right\|_T,
  \end{multline*}
  with the notations of  Proposition~\ref{propAB} and Lemma~\ref{lemC}.
From  Lemma~\ref{LemOm} we get  the relation, for $s{\ge}0$,
  \begin{multline*}
  \left|\croc{\Lambda^N_1(s),\Omega_{\Lambda^N_1(s)}(I)}{-}\croc{\Pi(s),\Omega_{\Pi(s)}(I)}\right|\\\leq (2\|\Psi\|_\infty{+}D_\Psi)\beta\|\Lambda^N_1(s){-}\Pi(s)\|_{\rm tv}{+} \left|\croc{\Lambda^N_1(s){-}\Pi(s),I}\right|.
  \end{multline*}
  By using Proposition~\ref{propAB} and Lemma~\ref{lemC}, we get that
\begin{multline}\label{inea}
d_2^N(t) =\E\left(\sup_{s\leq t}\left|\croc{\Lambda^N_1(s){-}\Pi(s),I}\right|\right)\leq
 \E\left(\left|\croc{\Lambda^N_1(0){-}\Pi(0),I}\right|\right)\\{+}    \left[(2\|\Psi\|_\infty{+}D_\Psi) \beta{+}4D_\Psi K\right]\int_0^t \E\left(\|\Lambda^N_1(s){-}\Pi(s)\|_{\rm tv}\right)\diff s\\
         {+}\left[\|\Psi\|_{\infty}{+}D_\Psi{+}1\right]\int_0^t \E\left(\left|\croc{\Lambda^N_1(s){-}\Pi(s),I}\right|\right)\,\diff s\\
+ D_\Psi\int_0^t \E\left(\croc{\Pi(s),I_K}\right)\,\diff s
 {+}\E\left(\sup_{s\le T}\left|M_{I,1}^N(s)\right|\right){+}\left\|Z^N\right\|_T.
\end{multline}
We now turn to the estimation of $d^N_1(t)$. For $f{:}\N{\to}\{0,1\}$ and  $T{>}0$,  by Lemma~\ref{LemOm},  if $t{\le}T$,
\[
   \int_0^t \left|\croc{\Lambda^N_1(s),\Omega_{\Lambda^N_1(s)}(f)}{-}\croc{\Pi(s),\Omega_{\Pi(s)}(f)}\right| \diff s
   \leq C_0\int_0^t \|\Lambda^N_1(s){-}\Pi(s)\|_{\rm tv}\diff s
\]
with $C_0{=}1{+}\beta(2\|\Psi\|_\infty{+}D_{\Psi})$,  Relations~\eqref{FP} and~\eqref{SDE4} give then the inequality 
 \begin{multline}\label{aq1}
   \left|\croc{\Lambda^N_1(t){-}\Pi(t),f}\right|\leq C_0\int_0^t \|\Lambda^N_1(s){-}\Pi(s)\|_{\rm tv}\diff s\\
   +\left\|\Lambda^N_1(0){-}\Pi(0)\right\|_{\rm tv}+\left|M_{f,i}^N(t)\right|{+} |X^N_{f,1}(t)|{+}|Y^N_{f,1}(t)|{+}|Z^N_{f,1}(t)|
 \end{multline}
 By definition of the total variation norm, see Relation~\eqref{TV},  we have 
\begin{equation}\label{bq1}
\left\|\Lambda^N_1(t){-}\Pi(t)\right\|_{\rm tv}=\sup_{f:\N\to\{0,1\}} \left|\croc{\Lambda^N_1(t){-}\Pi(t),f}\right|\leq R^N_1(t)+R^N_2(t),
\end{equation}
with
\[
R^N_1(t){\steq{def}}  \sup_{\mathclap{F\subset (K,+\infty)}} \left|\croc{\Lambda^N_1(t){-}\Pi(t),{\mathbbm 1}_F}\right|\quad \text{ and } \quad
R^N_2(t){\steq{def}} \sup_{\mathclap{F\subset [0,K]}} \left|\croc{\Lambda^N_1(t){-}\Pi(t),{\mathbbm 1}_F}\right|.
\]
With the same argument as before, we get
\begin{equation}\label{bq2}
 R^N_1(t)\leq  \croc{\Lambda^N_1(t){+}\Pi(t),(K,{+}\infty)}    \leq \left|\croc{\Lambda^N_1(t){-}\Pi(t),{\mathbbm 1}_{[0,K]}}\right| {+}2\croc{\Pi(t),I_K}   
\end{equation}
 and therefore
 \[
 R_1^N(t)\leq R_2^N(t)  {+}2\sup_{s\le t}\croc{\Pi(s),I_K}.
 \]
 Denote by ${\cal E}_K{=}\{f{=}{\mathbbm 1}_F{:}F{\subset}[0,K]\}$, By taking successively the supremum on all $f{\in}{\cal E}_K$ and $s{\leq} t$ for Relation~\eqref{aq1},  we obtain the inequality, for $t{\le}T$,
 \begin{multline}\label{eqd2}
\E\left( \sup_{s\le t} R^N_2(s)\right)\leq C_0\int_0^t\E\left(\sup_{u\le s}  \|\Lambda^N_1(u){-}\Pi(u)\|_{\rm tv}\right)\diff s\\
   {+}\E\left(\left\|\Lambda^N_1(0){-}\Pi(0)\right\|_{\rm tv}\right){+} 2^{{K+1}}\sup_{\mathclap{f\in{\cal E}_K}}\E\left(\sup_{s\le t} \left|M_{f,i}^N(s)\right|\right){+} \|X^N\|_t{+}\|Y^N\|_t{+}\|Z^N\|_T.
 \end{multline}
Again the quantity $\|X^N\|_t{+}\|Y^N\|_t$ is upper bounded with the help of  Proposition~\ref{propAB}. If we gather the estimates~\eqref{inea},  \eqref{bq1},  \eqref{bq2} and~ \eqref{eqd2}, we obtain that, for $T{\ge}0$, there exists a constant $C_0$ independent of $K{\ge}1$ and $T$ such that, for any $t{\le}T$,
 \begin{multline*}
   \frac{1}{C_0} d^N(t)\leq K\int_0^t d^N(s)\,\diff s + KT\sup_{s\le T}\croc{\Pi(s), I_K}
{+}\E\left(\left|\croc{\Lambda^N_1(0){-}\Pi(0),I}\right|\right)\\ {+}\E\left(\left\|\Lambda^N_1(0){-}\Pi(0)\right\|_{\rm tv}\right){+} 2^{K}\sup_{{f\in{\cal E}_K}{\cup}\{I\}}\E\left(\sup_{s\le t} \left|M_{f,i}^N(s)\right|\right){+}\left\|Z^N\right\|_T.
\end{multline*}
 Note that, for $s{\ge}0$,  $|\croc{\Lambda^N_1(s){-}\Pi(s),I}|{\le}F_N/N{+}\beta$,  the mapping  $s{\mapsto}d^N(s)$ is therefore bounded by a constant.   By using Assumption~\ref{CondI}, Lemmas~\ref{lemC} and~\ref{lemM}, and by applying Fatou's Lemma, we obtain the relation, 
 \[
\overline{d}(t)\leq    C_0 K\int_0^t \overline{d}(s)\,\diff s{+} KC_0T\, \E\left(\sup_{s\le t} \croc{\Pi(s), I_K}\right).
\]
The proposition is proved. 
\end{proof}

We can now prove the main convergence result. 

\begin{proof}[Proof of Theorem~\ref{theoLED}]
The notations of the proof of the previous proposition are used. With Gr\"onwall's Inequality and Relation~\eqref{IneqI}, we get the relation
\[
\overline{d}(t)\leq  C_0T K\E\left(\sup_{s\leq T} \croc{\Pi(s),I_K}\right)e^{C_0 Kt}
\]
for all $t{\le}T$.   Proposition~\ref{propExp} gives the existence of some $\eta{>}0$ and $C_1{>}0$ such that
\[
\E\left(\sup_{s\leq T} \croc{\Pi(s),I_K}\right)\leq C_1 e^{-\eta K}, \quad \forall K{>}0.
\]
By letting $K$ go to infinity, we obtain that $\overline{d}(t){=}0$ for all $t{\leq}t_1{\steq{def}}\eta/(2C_0){\wedge} T$ and therefore the desired convergence on the time interval $[0,t_1]$. 
In particular if we make a time shift at $t_1$, it is easy to check that Assumption~\ref{CondI} of Section~\ref{SecEvol} are also satisfied for this initial state and we can then repeat the same procedure until the time $T$ is reached. 
The theorem is proved. 
\end{proof}

\begin{theorem}[Mean-Field Convergence]\label{theoED}
  Under the conditions of Theorem~\ref{theoLED} and if the distribution of $\Pi(0)$ of Relation~\eqref{tvInit} of Assumption~\ref{CondI} is a Dirac mass at $\pi_0{\in}M_1(\N)$, then, for the convergence in distribution of processes,
  \[
  \lim_{N\to+\infty} \left(\frac{1}{N}\sum_{i=1}^N\delta_{L^N_i(t)}\right){=}(\Pi(t)),
  \]
  and, for any $p{\ge}2$, and $(n_i){\in}\N^p$, 
  \[
  \lim_{N\to+\infty} \left( \left(L^N_{n_1}(t)\right), \left(L^N_{n_2}(t)\right), \ldots, \left(L^N_{n_p}(t)\right)\right){\steq{dist}} Q_{\pi_0}^{\otimes p}
  \]
 where $Q_{\pi_0}$ is the distribution of the  of McKean{-}Vlasov process $(L(t))$, the solution of the SDE~\eqref{McKV} with $L(0){\steq{dist}}\pi_0$.
\end{theorem}
\begin{proof}
  Let $(\Lambda^N(t))$ be the process of the empirical distribution associated to the vector  $(L_i^N(t),1{\leq}i{\leq}N)$,
  \[
 \left( \Lambda^N(t)\right)\steq{def}\left(\frac{1}{N}\sum_{i=1}^N\delta_{L^N_i(t)}\right),
 \]
 then, for $t{\ge}0$,  it is easily seen that, because of the structure of the underlying graph,  the relation
 \[
 \Lambda^N(t)=\frac{1}{N}\sum_{i=1}^N \Lambda^N_i(t)
 \]
holds,  hence
\[
\sup_{t\leq T}   \left\|\Lambda^N(t){-}\Pi(t)\right\|_{\rm tv} \leq \frac{1}{N}\sum_{i=1}^N   \sup_{t\leq T}  \left\|\Lambda_i^N(t){-}\Pi(t)\right\|_{\rm tv}.
\]
The first claim of the theorem follows directly from Theorem~\ref{theoLED} and the  fact that the processes  $(\Lambda_i^N(t))$, $1{\le}i{\le}N$, have the same distribution. The last assertion is a simple consequence  of the fact that $(\Pi(t))$ is in this case a deterministic process and of Proposition~2.2 p.~177 of Sznitman~\cite{Sznitman}. 
\end{proof}

\subsection*{Convergence Properties of the Invariant Distributions}
We conclude this section with a mean-field convergence result for the invariant distributions of evolution equations~\eqref{FP}  when the sequence of the sizes $(h_N)$ of the neighborhoods grows at  linearly with respect $N$. Further results, Propositions~\ref{propPsi} and~\ref{PropConv}, will be given in Section~\ref{SecInv}. 

For a fixed $N$, the irreducible Markov process $(L_i^N(t))$ with a finite state space  has a unique invariant  distribution.  Let  $\widehat{L}^N{=}(\widehat{L}_i^N)$  be a random variable whose distribution is the invariant measure of $(L_i^N(t))$, and 
  \[
  \widehat{\Lambda}_1^N\steq{def} \frac{1}{N}\sum_{i\in {\cal H}_1} \delta_{\widehat{L}_i^N}
  \]
has the same distribution as  the local empirical distribution at node $1$ at equilibrium. 

\begin{theorem}[Asymptotic Behavior of Invariant Distributions]\label{TheoInv}
  Under the conditions of Theorem~\ref{theoLED}, and if
  \[
  \liminf_{N\to+\infty}{h_N}/{N}{>}0,
  \]
{  then the sequence $(\widehat{\Lambda}_1^N)$  is tight and the distribution of any of its limiting points $Q$ is  an invariant distribution for the dynamical system of Relation~\eqref{FP}, that is, if ${\rm Law}(\Pi(0)){=}Q$, then, for all $s{\in}\R_+$,  the process   $(\Pi(s{+}t), t{\ge}0)$ has the same distribution as $(\Pi(t), t{\ge}0)$.}
\end{theorem}
\begin{proof}
We first show that some  exponential moments of the variables $\widehat{L}_1^N$, $N{\ge}1$,  are bounded. 
For $\eta{>}0$,  the balance equation for the function $f(\ell){=}\exp(\eta \ell_1)$, $\ell{\in}{\cal S}_N$, gives the relation, remember that $L_1^N{\leq}F_N$, 
\[
  \E\left(e^{\eta \widehat{L}_1^N}\right){-}1{=}\E\left(e^{\eta \widehat{L}_1^N}\hspace{-1mm}\left[(e^\eta{-}1) \sum_{\mathclap{i\in{\cal H}^N(1)}} \widehat{L}_i^Np^N_{i1}\left(\widehat{L}^N\right){+}U_N\left(\widehat{L}^N\right)\right]\right),
\]
where, for $\ell{\in}{\cal S}_N$,
\[
U_N(\ell)=\sum_{i{\in}{\cal H}^N(1)}\left[1+p^N_{i1}(\ell)(e^\eta{-}1)\right]^{\ell_i} {-}1{-}\ell_ip^N_{i1}(\ell)(e^\eta{-}1).
\]
By using Relation~\eqref{Binom+}, we have $p^N_{i1}(\ell){\le}C_0/h_N$, we get then the estimation
\begin{multline*}
  |U_N(\ell)| \leq \sum_{i{\in}{\cal H}^N(1)} \left(e^\eta{-}1\right)^2\frac{C_0^2\ell_i^2}{2h_N^2}\left(1+\frac{C_0}{h_N}(e^\eta{-}1)\right)^{\ell_i-2}\\
{\leq}  \left(e^\eta{-}1\right)^2\frac{C_0^2}{2}\left(1{+}\frac{C_0}{h_N}(e^\eta{-}1)\right)^{F_N} \frac{F_N^2}{h_N^2},
\end{multline*}
since $\ell{\in}{\cal S}_N$. The assumption of the sequence $(h_N)$ shows that there exists some $\eta_1{>}0$ such that if $\eta{<}\eta_1$ then the last term is bounded by $D_0(e^\eta{-}1)^2$ for some constant $D_0{\ge}0$. By using this inequality, we get the relation
\begin{multline*}
  \E\left(e^{\eta \widehat{L}_1^N}\right){-}1 \leq \E\left(e^{\eta \widehat{L}_1^N}\left(\sum_{i\in{\cal H}^N(1)} \widehat{L}_i^N\frac{C_0}{h_N}\left(e^\eta{-}1\right){+}D_0(e^\eta{-}1)^2\right)\right)\\
\leq   \E\left(e^{\eta \widehat{L}_1^N}\left(D_1\left(e^\eta{-}1\right) {+}D_0(e^\eta{-}1)^2\right)\right),
\end{multline*}
for some constant $D_1{\ge}0$.  If $\eta_0{<}\eta_1$ is such that $D_1(e^{\eta_0}{-}1){+} D_0(e^{\eta_0}{-}1)^2{<}1$, then we get that the exponential moments of order $\eta_0$ of $L_1^N$ are bounded,
\begin{equation}\label{gq1}
\sup_{N{\ge} 1}\E\left(e^{\eta_0 \widehat{L}_1^N}\right)<{+}\infty.
\end{equation}
For $K{>}0$, 
\begin{equation}\label{gq2}
\E\left(\widehat{\Lambda}_1^N([K,{+}\infty))\right)=\P\left(\widehat{L}_1^N{\ge}K\right),
\end{equation}
{
from Lemma~3.2.8 of Dawson~\cite{Dawson}, we deduce that the sequence of local empirical distributions at equilibrium $(\widehat{\Lambda}_1^N)$ is  tight for the convergence in distribution. We take a subsequence $(N_k)$ so that  the sequence $(\widehat{\Lambda}_1^{N_k})$  converge to a random probability distribution on $\N$, $\widehat{\Pi}$.

  With the same argument as in the proof of  Theorem~\ref{theoED}, we have that the sequence of {\em global} empirical distribution $(\widehat{\Lambda}^{N_k})$ is also converging in distribution to $\widehat{\Pi}$. The relation of conservation of mass $\langle\widehat{\Lambda}^{N_k},I{\rangle}{=}F_{N_k}/N_k$, we thus get that, almost surely $\langle\widehat{\Pi},I\rangle{=}\beta$. The vector $\widehat{L}^{N_k}$ satisfies therefore the conditions of Assumption~\ref{CondI}.

  We define the process $(\widehat{L}_i^N(t))$ associated to the SDE~\eqref{SDE} with the initial point $(\widehat{L}_i^N)$ and $(\widehat{\Lambda}_1^{N}(t))$ the corresponding local empirical distributions.  Theorem~\ref{theoLED} shows the convergence in distribution
  \[
  \lim_{k\to+\infty} \left(\widehat{\Lambda}_1^{N_k}(t)\right)=\left(\Pi(t)\right),
  \]
  where $(\Pi(t))$ is the  solution of Equation~\eqref{FP} with initial point $\widehat{\Pi}$. The last assertion comes from the fact that for $s{\ge}0$ and $k{\in}\N$, by invariance, we have the equality 
  \[
 {\rm Law}\left[ \left(\widehat{\Lambda}_1^{N_k}(t), t{\ge}0\right)\right]=  {\rm Law}\left[\left(\widehat{\Lambda}_1^{N_k}(t{+}s), t{\ge}0\right)\right].
  \]
The proposition is proved.
}
  \end{proof}
The following section is devoted to the properties of the invariant distributions of the McKean{-}Vlasov process. 
\section{Equilibrium of the McKean{-}Vlasov Process}\label{SecInv}
Recall that $\beta$ is the average load of an urn at any time. The purpose of this section is to investigate the properties of the tail distribution of the number of balls in an urn at equilibrium when $\beta$ get large. Recall that in our asymptotic picture, the time evolution of the number of balls in a given urn can be seen as  the solution of the  SDE defining the  McKean{-}Vlasov process
\begin{equation}\label{McKV2}
\diff L_\beta(t) =\overline{\cal P}\left(\diff t{\times}\left[0,\beta\,\rule{0mm}{4mm} \Psi(\pi(t),L_\beta(t{-}))\right]\right){-}L_\beta(t{-}) {\cal P}(\diff t),
\end{equation}
where $\overline{\cal P}$ [resp. ${\cal P}$]  is an homogeneous Poisson point process on $\R_+^2$ [resp. $\R_+$] with rate $1$.

Suppose the process $(L_\beta(t))$ starts from equilibrium, {i.e.\   if $\pi_\beta{\in}M_1(\N)$ is the distribution of $L_{\beta}(0)$  then, for all $t{\ge}0$, the distribution $\pi(t)$ of $L_\beta(t)$ is constant and equal to $\pi_\beta$.} In this case  the process $(L_\beta(t))$ is a classical Markov jump process on $\N$. It is easily checked that any invariant distribution $\pi_\beta$ satisfies the following fixed point equation $F_\beta(\pi_\beta){=}\pi_\beta$ in $M_1(\N)$ where $F_\beta$ is defined by, for $\pi{\in}M_1(\N)$,
\begin{equation}\label{eqinv}
     F_\beta(\pi){=}\left( F_\beta(\pi)(n)\right)\steq{def}\left(\frac{1}{1{+}\beta\Psi(\pi,n)}\prod_{k=0}^{n-1} \frac{\beta\Psi(\pi,k)}{1+\beta\Psi(\pi,k)}\right),
\end{equation}
with the convention that $\prod_0^{{-}1}{=}1$. Clearly $F_\beta(\pi){\in}M_1(\N)$ and, due to Relation~\eqref{Psimass}, the quantity
\[
\sum_{n=0}^{+\infty}\prod_{k=0}^{n} \frac{\beta\Psi(\pi,k)}{1+\beta\Psi(\pi,k)}
\]
 is $\beta$ when $\pi$ is a fixed point of $F_\beta$.

The following proposition shows that, for a large class of functionals $\Psi$,  a potential invariant distribution is always concentrated around  values proportional to  its average $\beta$. We will give more precise results later for power of $d${-}choices algorithms. Additionally, an existence and uniqueness result is proved when the average load per urn is small enough. 
\begin{prop}[Existence and Uniqueness  of Invariant Measure]\label{propPsi}\ \\
If  Assumption~\ref{CondA} of Section~\ref{subsecallocation} holds for the functional $\Psi$,
\begin{enumerate}
\item if, for $\beta{>}0$, the distribution of a random variable $Y_\beta$ is  invariant for the McKean{-}Vlasov process~\eqref{McKV}, then $Y_\beta$ is stochastically bounded by a geometric distribution with parameter $\beta\|\Psi\|_\infty/(1{+}\beta\|\Psi\|_\infty)$, in particular, for any $x{>}0$,
\begin{equation}\label{tailinv}
\P\left(\frac{Y_\beta}{\beta}\geq x\right){\le}\exp\left(-\frac{\beta}{1{+}\beta\|\Psi\|_\infty}x\left(1-\frac{1}{2(1{+}\beta\|\Psi\|_\infty)}\right)\right).
\end{equation}
{
\item There exists $\beta_0{>}0$ such that, if $\beta{<}\beta_0$, the McKean{-}Vlasov process~\eqref{McKV} $(L(t))$ has a unique  invariant distribution $\pi_\beta$, and there exist  $\gamma_\beta{>}0$ such that, for any initial distribution $\pi_0{\in}M_1(\N)$ with a second moment, 
  \begin{equation}\label{expcv}
    W_2(\pi(t),\pi_\beta)\leq W_2(\pi_0,\pi_\beta)e^{-\gamma_\beta t},\quad \forall t{\ge}0,
  \end{equation}
 where $(\pi(t)){=}(\P(L(t){=}\cdot))$ and $W_2(\cdot,\cdot)$ is the Wasserstein distance defined by Relation~\eqref{Wp}.
}  \end{enumerate}
\end{prop}
{It should be noted that for non-linear Markov processes, Inequalities like~\eqref{expcv} are, in general, delicate to obtain, even when there is a unique invariant distribution. See for example, Carillo et al.~\cite{Carillo} for non-linear diffusions associated to Langevin evolution equation, or Caputo et al.~\cite{Caputo} in a discrete state space setting. In our case a convenient coupling  and some calculus give the desired exponential decay of the Wasserstein distance to equilibrium.}
\begin{proof}
Let $\pi_\beta$ be the distribution of $Y_\beta$, If the initial distribution of the McKean{-}Vlasov process is $\pi_\beta$, then $(L_\beta(t))$ is a simple Markov process which jumps from $n$ to $n{+}1$ at rate $\beta\Psi(\pi_\beta,n)$ and returns at $0$ with rate $1$. Denote by $(\widetilde{L}_\beta(t))$  a Markov process on $\N$ with the same characteristics except that the jumps from $n$ to $n{+}1$ occur at rate $\beta\|\Psi\|_\infty$. It is easy to construct a coupling such that $\widetilde{L}_\beta(0){=}{L}_\beta(0){=}Y_\beta$ and that ${L}_\beta(t){\le}\widetilde{L}_\beta(t)$ holds for all $t{\ge}0$.  By assumption, the distribution of ${L}_\beta(t)$ is constant with respect to $t$. It is easy to check that the invariant distribution of $(\widetilde{L}_\beta(t))$ is a geometric distribution with parameter $\beta\|\Psi\|_\infty/(1{+}\beta\|\Psi\|_\infty)$. This proves the first part of the proposition. 

For $k{\in}\N$ and $\pi{\in}M_1(\N)$, one defines
  \[
a_\pi(k){\steq{def}}\frac{\beta\Psi(\pi,k)}{1{+}\beta\Psi(\pi,k)},
\]
the function $F_\beta$ can be expressed as, for $n{\in}\N$,
\[
F_\beta(\pi)(n)=\prod_0^{n-1}a_\pi(k){-}\prod_0^{n}a_\pi(k).
\]
From Relations~\eqref{LipPsi}  and the boundedness of $\Psi$, we get that
\[
\|a_\pi{-}a_{\pi'}\|_{\infty}\leq \beta D_\Psi \|\pi{-}\pi'\|_{\rm tv} \quad\text{ and }\quad a_\pi(k)\leq \delta\steq{def}\frac{\beta\|\Psi\|_\infty}{1{+}\beta\|\Psi\|_\infty}
\]
hold for any $\pi$, $\pi'{\in}M_1(\N)$ and $k{\in}\N$. Note that, for $n{\in}\N$,
\[
\prod_0^{n}a_{\pi'}(k){-}\prod_0^{n}a_{\pi}(k)
=\sum_{m=0}^{n}\prod_0^{m}a_{\pi}(k)\prod_{m+1}^{n}a_{\pi'}(k){-}\prod_0^{m-1}a_{\pi}(k)\prod_{m}^{n}a_{\pi'}(k),
\]
and therefore that, for $n{\ge}1$,
\[
  \left|\prod_0^{n}a_{\pi'}(k){-}\prod_0^{n}a_{\pi}(k)\right|
 \leq \sum_{m=0}^{n}\left|a_\pi(m){-}a_{\pi'}(m)\right|\delta^{n}\leq (n{+}1)\delta^{n}\beta D_\Psi \|\pi{-}\pi'\|_{\rm tv}.
\]
This gives the Lipschitz property for the mapping $F_\beta$ for the total variation norm, 
\begin{align*}
  \|F_\beta(\pi'){-}F_\beta(\pi)\|_{\rm tv}&=\frac{1}{2}\sum_{n=0}^{+\infty} |F_\beta(\pi')(n){-}F_\beta(\pi)(n)|\\
  & \leq \sum_{n=0}^{+\infty} \left|\prod_0^{n}a_{\pi'}(k){-}\prod_0^{n}a_{\pi}(k)\right|
\leq \frac{\beta D_\Psi}{(1{-}\delta)^2} \|\pi'{-}\pi\|_{\rm tv}. 
\end{align*}
Hence if $\beta D_\Psi(1{+}\beta\|\Psi\|_\infty)^2{<}1$, $F_\beta$ is a contracting application for the total variation norm. Hence we get the existence and uniqueness of an invariant distribution.

{
  Let $\sigma_a$ and $\sigma_b{\in}M_1(\N)$, two probability distributions on $\N$ with a second moment, and $A$ and $B$ two integer valued random variables whose distributions are respectively $\sigma_a$ and $\sigma_b$. We denote 
$(L_a(t))$ and $(L_b(t))$ the solutions of the McKean-Vlasov equation~\eqref{McKV} such that $L_a(0){=}A$ and $L_b(0){=}B$ and, for $c{\in}\{a,b\}$,
\[
\diff L_c(t) =\overline{\cal P}\left(\diff t{\times}\left[0,\beta\,\rule{0mm}{4mm} \Psi(\pi_c(t),L(t{-}))\right]\right){-}L_c(t{-}) {\cal P}(\diff t),
\]
where $\pi_c(t)$ is the distribution of $L_c(t)$. Note that for both processes we use the same Poisson processes $\overline{\cal P}$ and ${\cal P}$ with intensity $1$ on $\R_+^2$ and $\R_+$ respectively. It is not difficult to show that, for all $t{\ge}0$ and $c{\in}\{a,b\}$, $L_c(t)$ has a second moment.

By using the SDE's, we get that
\begin{multline*}
(L_a(t){-}L_b(t))^2\le(A{-}B)^2{-}\int_0^t\left(L_a(s{-}) {-}L_b(s{-}) \right)^2{\cal P}(\diff s)\\
{+}\int_0^t\left(2\left|L_a(s{-}){-}L_b(s{-})\right|+1\right) \overline{\cal P}\left(\diff s{\times}J(s)\right)
\end{multline*}
where $J(s)$ is the intervals whose end points are $\beta\,\Psi(\pi_c(s),L_c(s{-}))$, $c{\in}\{a,b\}$. From  Assumption~\ref{CondA} of Section~\ref{SecMod}, there exists a constant $C$ such that
\[
|J(s)|\steq{def}\int_{J(s)}\diff u\leq \beta C\left(|L_a(s{-}){-}L_b(s{-})|+\|\pi_a(s){-}\pi_b(s)\|_{tv}\right).
\]
By taking the expectation in the last inequality we obtain that, for $t{\ge}0$,
\begin{multline*}
\E\left[(L_a(t){-}L_b(t))^2\right]\le\E\left[(A{-}B)^2\right]{-}\int_0^t\E\left[\left(L_a(s) {-}L_b(s) \right)^2\right]\,\diff s\\
{+}\beta C\int_0^t\left(2\left|L_a(s){-}L_b(s)\right|+1\right)\left(|L_a(s){-}L_b(s)|+\|\pi_a(s){-}\pi_b(s)\|_{tv}\right) \diff s
\end{multline*}
 Since the processes are integer-valued, by using that $|x|{\le}x^2$ for $x{\in}\Z$, we have, as in the proof of Theorem~\ref{theoMcKV}, for $s{\ge}0$, 
\[
\|\pi_a(s){-}\pi_b(s)\|_{tv}{\le}W_1(\pi_a(s),\pi_b(s))\leq W_2(\pi_a(s),\pi_b(s))^2.
\]
and with the fact that the total variation distance is bounded by $1$,
\begin{multline*}
  \int_0^t\left(2\left|L_a(s){-}L_b(s)\right|+1\right)\left(|L_a(s){-}L_b(s)|+\|\pi_a(s){-}\pi_b(s)\|_{tv}\right) \diff s
\\  \le   5\int_0^t \left(L_a(s){-}L_b(s)\right)^2\diff s
  +\int_0^t W_2(\pi_a(s),\pi_b(s))^2\,\diff s.
\end{multline*}
We thus get that 
\begin{multline*}
\E\left[(L_a(t){-}L_b(t))^2\right]+(1{-}5\beta C)\int_0^t\E\left[\left(L_a(s) {-}L_b(s) \right)^2\right]\,\diff s  \le\\ \E\left[(A{-}B)^2\right]
{+}\beta C\int_0^tW_2(\pi_a(s),\pi_b(s))^2 \diff s
\end{multline*}
holds. If we choose $\beta$ so that $\gamma{=}(1{-}6\beta C)/2{>}0$, it gives
\begin{multline*}
W_2(\pi_a(t),\pi_b(t))^2+(1{-}5\beta C)\int_0^tW_2(\pi_a(s),\pi_b(s))^2\,\diff s  \\\le \E\left[(A{-}B)^2\right]
{+}\beta C\int_0^tW_2(\pi_a(s),\pi_b(s))^2 \diff s.
\end{multline*}
Gr\"onwall's Inequality gives the relation
\[
W_2(\pi_a(t),\pi_b(t))\leq e^{-\gamma t}\sqrt{\E\left[(A{-}B)^2\right]}.
\]
By taking the minimum on all couplings $(A,B)$ with marginals $\sigma_a$ and $\sigma_b$, we get finally  the inequality 
\[
W_2(\pi_a(t),\pi_b(t))\leq e^{-\gamma t}W_2(\pi_a(0),\pi_b(0))
\]
Hence if $\beta$ is sufficiently small, Equation~\eqref{tailinv} shows that the unique invariant distribution has a second moment. Inequality~\eqref{expcv} is then a consequence of the last relation. The proposition is proved. }
\end{proof}

{
We can now give a stronger version of Proposition~\ref{TheoInv} when $\beta$ is sufficiently small.  Recall that
 $\widehat{L}^N{=}(\widehat{L}_i^N)$  is a random variable whose distribution is the invariant measure of $(L_i^N(t))$, and 
  \[
  \widehat{\Lambda}_1^N\steq{def} \frac{1}{N}\sum_{i\in {\cal H}_1} \delta_{\widehat{L}_i^N}
  \]
has the same distribution of  the local empirical distribution at node $1$ at equilibrium. 

\begin{prop}[Convergence of Invariant Distributions]\label{PropConv}
Under the conditions of Theorem~\ref{theoLED}, and if
  \[
  \liminf_{N\to+\infty}{h_N}/{N}{>}0,
  \]
there exists some $\beta_0{>}0$ such that if $\beta{\le}\beta_0$, then the sequence $(\widehat{\Lambda}_1^N)$ [ resp. $(\widehat{L}_1^N)$]  is converging in distribution to $\delta_{\pi_\beta}$ [resp. $\pi_\beta$], where $\pi_\beta$ is the unique invariant distribution of the McKean{-}Vlasov process. 
\end{prop}
\begin{proof}
From Theorem~\ref{TheoInv}, we know that the sequence  $(\widehat{\Lambda}_1^N)$ is tight the distribution of any limiting point $Q$ is an invariant distribution of Relation~\eqref{FP}. Let $(\Pi(t))$ the corresponding process with $Q$ as the distribution of $\Pi(0)$, we know it is stationary. Inequality~\eqref{gq1} shows that $Q$ has a finite  exponential moment and, therefore, a second moment. Relation~\eqref{expcv} gives, for $\beta$ sufficiently small and $t{\ge}0$,
  \[
  \E\left({W}_2(\Pi(t),\pi_\beta)\right)\leq e^{-\gamma_\beta t} \E\left({W}_2(\Pi(0),\pi_\beta)\right),
  \]
from  the stationarity property of $(\Pi(t))$ we get
  \[
  \E\left({W}_2(\Pi(t),\pi_\beta)\right)=  \E\left({W}_2(\Pi(0),\pi_\beta)\right),
  \]
  consequently $\E({W}_2(\Pi(0),\pi_\beta)){=}0$, $Q$ the distribution of $\Pi(0)$ is therefore the Dirac mass at $\pi_\beta$,  the sequence  $(\widehat{\Lambda}_1^N)$ is converging in distribution to $\delta_{\pi_\beta}$
This implies that, for any bounded function $f$ on $\N$, we have
  \[
  \lim_{N\to+\infty} \E\left(\widehat{\Lambda}_1^N\right)=\croc{\pi_\beta,f},
    \]
    and the symmetry property gives $\E(\widehat{\Lambda}_1^N(f)){=}\E(f(\widehat{L_1}^N))$. The convergence in distribution of $(\widehat{L_1}^N)$ to $\pi_\beta$  is proved.
\end{proof}
  }
{  
  As it will be seen for specific functionals $\Psi$ much more can be said, in particular the existence and uniqueness of the invariant distribution of the McKean{-}Vlasov process {\em for all} $\beta{>}0$. The rest of this section is devoted to the Random Weighted Algorithm and the Power of $d${-}choices Algorithm,  the existence and uniqueness of the invariant measure of the  McKean{-}Vlasov process defined by SDE~\eqref{McKV} is established and its limiting behavior when the load $\beta$ is large is investigated.}

\subsection{Random Weighted Algorithm}
For $\sigma{\in}M_1(\N)$, the function $\Psi$ is, in this case, defined by
\[
  \Psi_{\rm cc}(\sigma,l)=\frac{W(l)}{\croc{\sigma,W}},
  \]
  and  the range of $W$ is assumed to be in $[c,C]$, for some positive constants $c$ and $C$.

  It is not difficult to see that there exists a unique invariant distribution $\pi_\beta$. If $\pi{\in}M_1(\N)$ is invariant if and only if  it  has  the representation 
  \[
(\pi(n))=  \left(\frac{\gamma}{\gamma {+}\beta W(n)}\prod_{k=0}^{n-1} \frac{\beta W(k)}{\gamma{+}\beta W(k)}\right),
  \]
  where $\gamma{=}\croc{\pi,W}$, i.e.\ 
\begin{equation}\label{InvCC}
\beta =\sum_{n=0}^{+\infty}\prod_{k=0}^{n} \frac{\beta W(k)}{\gamma{+}\beta W(k)}.
\end{equation}
It is easy to see that this equation has a unique solution $\gamma_\beta$ which gives the existence and the uniqueness of the invariant distribution in this case. 

When $(W(k))$ is constant equal to $w{>}0$,  then balls are placed at random in the neighboring urns, independently of their loads in particular. In this case $\gamma_\beta{=}w$ and the corresponding invariant measure is the geometric distribution with parameter $\beta /(1{+}\beta )$. 
\begin{prop}
For the Random Algorithm and any $\beta{>}0$, the geometric distribution with parameter $\beta /(1{+}\beta )$ is the equilibrium distribution of the McKean{-}Vlasov process defined by SDE~\eqref{McKV}, if $Y_\beta$ is a random variable with such a distribution, for the convergence in distribution, 
\[
\lim_{\beta\to+\infty}\frac{Y_\beta}{\beta}=E_1,
\]
where $E_1$ is an exponentially distributed random variable with parameter $1$.
\end{prop}
It turns out that the random algorithm behaves poorly in terms of the load of a given urn.  For a large $\beta$, the asymptotic tail distribution of the occupancy of an urn at equilibrium is, as expected, the upper bound~\eqref{tailinv}.  The simple consequence of this result is that, if on average, there are $\beta$ balls per urn, there is a significant fraction of urns with an arbitrarily large number of balls.  We will see that the situation is completely different for the power of $d${-}choices algorithm.

\subsection{Power of $\mathbf{d}${-}choices}
For $\sigma{\in}M_1(\N)$, the function $\Psi$ is, in this case, 
\[
\Psi_{\rm pc}(\sigma,l){=}\frac{(\sigma([l,{+}\infty))^d{-}\sigma((l,{+}\infty))^d}{\sigma(\{l\})}.
  \]
\begin{prop}\label{propPCInv}
For any $\beta{>}0$, there exists a unique invariant distribution $\pi_\beta$ of $(L_\beta(t))$, it is given by
\begin{equation}\label{fq1}
\pi_\beta(n)=\xi_n{-}\xi_{n+1},\quad n{\ge}0,
\end{equation}
where $(\xi_n)$ is the non-increasing sequence defined by induction by $\xi_0{=}1$ and
\begin{equation}\label{eqpc}
  \xi_n{=}\beta\left(\xi_{n-1}^d{-}\xi_n^d\right), \quad n{\ge}1.
\end{equation}
{For any $\sigma\in M_1(\N)$ with finite second moment, if $\sigma(t)$ is the distribution of the solution of McKean{-}Vlasov Equation~\eqref{McKV} at time $t$  with initial distribution $\sigma$, then  we have
\begin{equation}\label{cvpc}
  \lim_{t\to\infty}\|\sigma(t){-}\pi_\beta\|_{\rm{tv}}=0.
\end{equation}
}
\end{prop}
{
  The convergence in distribution of the McKean{-}Vlasov to its invariant measure obtained in Proposition~\ref{propPsi} in a general setting is with the assumption that $\beta$ is sufficiently small. Note that a related convergence~\eqref{cvpc} for the Power of Choices Algorithm is obtained for all $\beta{>}0$.
  }
\begin{proof}
The existence and uniqueness of such a sequence $(\xi_n)$ is clear with standard calculus.  Let $\pi$ be an invariant distribution of $(L_\beta(t))$, it satisfies the balance equation
\[
\pi(n)(1+\beta\Psi(\pi,n))=\pi(n{-}1)\beta\Psi(\pi,n{-1}),\quad n{\ge}1.
\]
and   $\pi(0)\beta\Psi(\pi,0){=}1{-}\pi(0)$.
Define $\xi_n=\pi([n,{+}\infty))$. The balance equation for $n{=}0$ shows that Relation~\eqref{eqpc} is clearly true for  $n{=}1$. If we assume that Relation~\eqref{fq1} holds up to $p$,
the balance equation can be rewritten as
  \[
(\xi_{p}{-}\xi_{p+1})+\beta\left(\xi_{p}^d{-}\xi_{p+1}^d\right)=\beta\left(\xi_{p-1}^d{-}\xi_p^d\right)=\xi_p,
  \]
hence Relation~\eqref{eqpc} is valid for $p{+}1$.  The first part of proposition is proved. 

{
Now let $\eta_n(t){=}\sigma(t)[n,\infty)$, from Relation~\eqref{McKV}, it is easy to see that $(\eta_n(t),n\ge 1)$ satisfies the following ODE, for $n{>}1$ and $t{\in}\R_+$,
\[
(\eta_n(t))'=\beta\left(\eta_{n-1}(t)^d{-}\eta_{n}(t)^d\right)-\eta_n(t).
\]
By using the recurrence relation defining the sequence $(\xi_n)$, we obtain
\begin{multline*}
e^{-t}\left(e^{t}\left(\eta_n(t){-}\xi_n\right)^2\right)'=
2\beta \left(\eta_n(t){-}\xi_n\right)\left(\eta_{n-1}(t)^d{-}\xi_{n-1}^d\right)\\
-2\beta \left(\eta_n(t){-}\xi_n\right)\left(\eta_n(t)^d{-}\xi_n^d\right)-\left(\eta_n(t){-}\xi_n\right)^2.
\end{multline*}
For $n{=}1$ this gives $(\eta_1(t){-}\xi_1)^2{\le}e^{-2t}$. By using the inequality $2\beta x y{\le}\beta^2x^2{+}y^2$, for $x$, $y{\in}\R$ and  that $|x^d{-}y^d|{\le}d|x{-}y|$ holds for $x$, $y{\in}[0,1]$, we get 
\[
e^{-t}\left(e^{t}\left(\eta_n(t){-}\xi_n\right)^2\right)'\le \beta^2d^2\left(\eta_{n{-}1}(t){-}\xi_{n-1}\right)^2,
\]
and therefore the relation
\[
e^t\left(\eta_n(t){-}\xi_n\right)^2\le (\beta d)^2\int_0^t e^s\left(\eta_{n-1}(s){-}\xi_{n-1}\right)^2\diff s +1,
\]
for all $n\ge 2$ and, by induction,
\begin{equation}\label{eeq1}
e^t\left(\eta_n(t)-\xi_n\right)^2\le (1+(d\beta)^2)\sum_{k=0}^{n-2}\frac{(d\beta)^{2k}t^k}{k!}.
\end{equation}
When $\beta{<}1/d$, this gives
\[
\sup_n\left|\eta_n(t){-}\xi_n\right|\le \sqrt{2} \exp{\left(((d\beta)^2{-}1)t/2\right)},
\]
and therefore the exponential stability in total variation distance. For $\beta{\ge} 1/d$ and $n{\ge}1$,  Inequality~\eqref{eeq1} gives the relation $|\eta_n(t){-}\xi_n|{\le}(\sqrt{2}d\beta)^{n-1}\exp({-}t/4)$, hence 
\[
\frac{1}{2}\sum_{k=0}^{n-1}\left|\sigma_k(t){-}\pi_\beta(k)\right|\le 
\sum_{k=0}^{n-1}(\sqrt{2}d\beta)^{k}e^{-t/4}=\frac{(\sqrt{2}d\beta)^n-1}{\sqrt{2}d\beta-1}e^{-t/4}.
\]
We conclude by using the elementary inequality
\[
\|\sigma(t){-}\pi_\beta\|_{\rm{tv}}\le \frac{1}{2}\sum_{k=0}^{n-1}\left|\sigma_k(t)-\pi_{\beta}(k)\right|+\frac{1}{2}|\eta_n(t)-\xi_n|+\xi_n.
\]
The proposition is proved. }
\end{proof}
The following theorem shows that the power of $d${-}choices policy is efficient in terms of the load of an arbitrary urn, the invariant distribution of this load is asymptotically concentrated on the finite interval $[0,d/(d{-}1)\beta]$. Only an extra  capacity $\beta/(d{-}1)$ has to be added  to the minimal capacity $\beta$ for any urn  in order to handle properly this allocation policy. This has important algorithmic consequences in some contexts. See Sun et al.~\cite{SSMRS2}.
\begin{theorem}[Power of $d${-}choices]\label{theoInv}
If, for  $\beta{>}0$, $Y_\beta$ is a random variable whose distribution is the unique invariant measure of the McKean{-}Vlasov process  $(L_\beta(t))$  defined by SDE~\eqref{McKV} then, for the convergence in distribution,
  \[
  \lim_{\beta\to+\infty}\frac{Y_\beta}{\beta}=\frac{d}{d{-}1}\left(1{-}U^{d-1}\right),
  \]
  where $U$ is a uniform random variable on $[0,1]$. 
\end{theorem}
\begin{proof}
  For $k{\ge}1$, by summing up Equation~\eqref{eqpc} for $1$ to $k{-}1$, one gets the relation
  \[
  \beta\left(1-\P(Y_\beta\geq k)^d\right)=\E\left(Y_\beta{\wedge}k\right),
  \]
  since $\E(Y_\beta){=}\beta$, this gives
  \[
  \P(Y_\beta\geq k)^d=\frac{1}{\beta}\left(\E(Y_\beta){-}\E\left(Y_\beta{\wedge}k\right)\right)=\E\left(\left(\frac{Y_\beta}{\beta}{-}\frac{k}{\beta}\right)^+\right),
  \]
  hence, if $Z_\beta{\steq{def}} Y_\beta/\beta$,  for $x{>}0$,
  \[
  \P\left(Z_\beta\geq x\right)^d=\E\left(\left(Z_\beta{-}\frac{\lceil \beta x\rceil }{\beta}\right)^+\right). 
  \]
  Relation~\eqref{tailinv} shows that, for $\beta_1$ sufficiently large  the family of  random variables $Z_\beta$, $\beta{\ge}\beta_1$  is tight  and that the corresponding second moments are bounded. Let $Z$ be a limiting point when $\beta$ gets large and  $h(x){\steq{def}}\P(Z{\ge}x)$, $x{\ge}0$. The previous relation gives the identity, for $ x{\geq 0}$,
  \[
h(x)^d{=}\int_x^{+\infty} h(u)\,\diff u. 
\]
It is easy to see that this relation determines $h$ and therefore gives the desired convergence in distribution. The theorem is proved.
\end{proof}

\providecommand{\bysame}{\leavevmode\hbox to3em{\hrulefill}\thinspace}
\providecommand{\MR}{\relax\ifhmode\unskip\space\fi MR }
\providecommand{\MRhref}[2]{%
  \href{http://www.ams.org/mathscinet-getitem?mr=#1}{#2}
}
\providecommand{\href}[2]{#2}

\end{document}